\documentclass[12pt]{article}
\usepackage{amsmath,amsthm,amssymb}

\newtheorem{theorem}{Theorem}
\newtheorem{lemma}[theorem]{Lemma}
\newtheorem{corollary}[theorem]{Corollary}
\newtheorem{claim}[theorem]{Claim}

\newtheorem{observation}[theorem]{Observation}

\newcommand\wpn{{\mathop{\rm wpn}}}
\newcommand\Forb{{\mathop{\rm Forb}}}
\newcommand\Co{{\mathop{\rm Co}}}

\title{The global structure of a typical graph without $H$ as an induced subgraph when $H$ is a cycle.} 
\author{B. Reed\thanks{Institute of Mathematics, Academia Sinica, Taiwan; \texttt{bruce.al.reed@gmail.com}.}}
\begin{document}
\maketitle  

\begin{abstract}
One way to certify that a graph does not contain an induced cycle of length six is to provide a partition of its
vertex set into  (i) a stable set, and (ii) a graph containing no stable set of size three and no induced matching of size two.   
We show that   almost every  graph  which  does not contain a cycle of length six as an induced subgraph
has such a certificate. We obtain similar characterizations of the structure of almost all  graphs which contain no induced cycle
of  length  $k$ for all  even $k$ exceeding six.  (Similar results were obtained for  $k=3$  by 
Erd\H os, Kleitman, and Rothschild in 1976,  for  $k =4,5$ by Pr\"omel and  Steger in 1991 and for odd $k$  exceeding 5 by Balogh and Butterfield in 2009.)  
We prove that a simiiar theorem for all $H$ holds up to the deletion of a set of $o(|V(G)|)$ vertices and ask for which $H$ the chartacterizattion holds 
fully.

\end{abstract}

\section{An Overview}

We say that a graph $G$ is {\em $H$-free} if $G$ does not contain an induced copy of $H$.

We discuss a possible characterization of  the structure of  typical $H-free$ graphs,  survey previous results proving the characterization 
holds   for certain special $H$, and prove the characterization for even cycles of length at least six, which combined with previous results shows that it holds 
for all cycles and their complements. In work subsequent to this, Reed and Yuditsky proved that the structural characterization holds for all trees.

An earlier version of this manuscript conjectured that the characterization  holds for  all $H$. This was disproven by Norine and Yuditsky\cite{NY}, so the conjecture has been removed. 
In this paper we show that that  for every $H$ there is an $\epsilon>0$ such that  the conjecture holds after the deletion of a set of $O(n^{1-\epsilon})$  exceptional vertices and ask 
 if we can improve on the upper bound on the size of this exceptional set (in general or for specfic $H$).

An early result in this area  was obtained by 
Erd\H os, Kleitman, and Rothschild\cite{EKR} over 35 years ago. They  characterized the structure of   typical graphs without a cycle of 
length three precisely: the vertex set of almost every $C_3$-free
graph is bipartite.

It was not until 1991, however, that Pr\"omel and Steger \cite{PS91,PS92,PSBerge,PS93}   suggested an approach that could be applied to all $H$. Using this approach, they showed that  almost every $C_4$-free graph is a split graph
(a graph is {\em split} if its vertex set can be partitioned into a 
stable set and a clique).  They also showed  that the vertex set of almost every  $C_5$-free graph can be partitioned into either 
(i) a clique and the disjoint union of cliques or (ii) a stable set and a complete multipartite graph. In addition, for every $H$, they  obtained upper  and lower bounds on the number of  
$H$-free graphs on $n$ vertices whose exponents differed by a multiplicative factor of $1+o(1)$. 

Subsequently, Balogh and Butterfield\cite{BB11} proved that for $k \ge 4$ almost every  $C_{2k+1}$-free graph can be partitioned into $k$ cliques while almost every $C_7$ free graph  can be partitioned either into 3 cliques or a stable set and 2 cliques.    

In Section \ref{C6}, we prove that almost every $C_6$-free graph is the disjoint union of 
a stable set and a graph containing no stable set of size three and no induced matching with two edges. In Section \ref{C2l} we present similar results for longer even cycles, which we will state precisely below. 
Before doing so, we discuss a characterization of   the structure of typical $H$-free graphs which holds for many $H$. To do so we need some definitions.

By $\Forb_H$ we mean the family of $H$-free graphs, and by $\Forb_H^n$ the family of $H$-free graphs on 
the vertex set $\{1,\dots,n\}$, which we denote $V_n$. The  {\em witnessing partition number} of $H$, denoted $\wpn(H)$, is  the  maximum $k$ such that for some pair $(c_H,s_H)$ which  sum to $k$,  
there is no partition of $V(H)$ into $c_H$ cliques and $s_H$ stable sets.

We define an {\em H-freeness witnessing  $k$-sequence} to be a sequence ${\cal S}=({\cal F}_1,\dots,{\cal F}_k)$ of   hereditary   
families of graphs such that for every partition $X_1,\dots,X_k$ of the vertices of $H$ into $k$ sets,  there is an $i$ such that $H[X_i]$ is not an element of  ${\cal F}_i$. 
For any such family and any $G$ for which $V(G)$ can be partitioned into $X_1,\dots, X_k$  such that $G[X_i]$ is in ${\cal F}_i$  for each $i$ we see that $G$ is $H$-free. 
We say that $X_1,\dots, X_k$ is an {\em $H$-freeness  witnessing partition for $G$, certified by ${\cal F}_1,\dots, {\cal F}_k$}.

Let ${\cal S}=({\cal F}_1,\dots,{\cal F}_k)$ be a ($H$-freeness) witnessing $k$-sequence.
If we set  
\begin{equation}\label{jdef}
J_i=\{ L: \mbox{$L$ is an induced subgraph of $H$ but not in ${\cal F}_i$}\}
\end{equation}
and let ${\cal F}_i^\prime$ be 
the set of graphs which have no induced subgraph in $J_i$ then ${\cal S}^\prime = ({\cal F}_1^\prime,\dots,{\cal F}_k^\prime)$  
is also an $H$-freeness witnessing sequence.   We  observe that any witnessing partition certified by ${\cal S}$ is also certified by ${\cal S}^\prime$.  
We say a witnessing sequence is {\em canonical} if for each ${\cal F}_i$ there is some  set $J_i$ of induced subgraphs of $H$ 
such that ${\cal F}_i$ consists of those graphs not containing any graph in $J_i$ as an induced subgraph.  Note that our observation shows 
that we may restrict our attention to canonical witnessing sequences.

Every $H$-{\em free} graph has an  $H$-freeness witnessing partition  with all but the last  part empty, 
certified by a seuqence where every family is empty 
except the last which consists of  all the $H$-free graphs. We are interested in partitions certified by families 
where each ${\cal F}_i$ contains arbitrary large graphs. Ramsey theory tells us this is equivalent to requiring that 
each contains either all cliques or all stable sets. We call such sequences {\it really canonical}. 

\begin{claim}
If ${\cal S}=({\cal F}_1,\dots,{\cal F}_{\wpn(H)})$ is a   really canonical witnessing $\wpn(H)$-sequence,
then  every $J_i$ contains a 
bipartite  graph, the complement of a bipartite graph and a split graph.
\end{claim}

\begin{proof}
Suppose first that $J_i$ contains no bipartite graph: reindexing we can assume $i=\wpn(H)$.
For each $j$ not equal to  $\wpn(H)$, $J_j$ contains either no cliques or no stable sets.
We let $c$ be the number of $j$ in $\{1,\dots,\wpn(H)-1\}$ such that $J_j$ contains no cliques,
and set $s=\wpn(H)+1-c\ge 2$.
By reindexing we may assume that $J_1,\dots,J_c$ all contain no cliques
Thus $\mathcal F_1,\dots,\mathcal F_c$ each contain all cliques,
and similarly $\mathcal F_{c+1},\dots,\mathcal F_{\wpn(H)-1}$ each contain all stable sets.  
Furthermore, $\mathcal F_{\wpn(H)}$ contains all bipartite graphs, as $J_{\wpn(H)}$ contains no bipartite graphs.

By the definition of
$\wpn(H)$ the vertex set of $H$ can be partitioned  into cliques 
$K_1,\dots,K_c$  and stable sets $Y_{c+1},\dots,Y_{\wpn(H)+1}$. 
For $j$ at most $c$, we let $X_j$ be $V(K_j)$.
For $j$ between $c+1$ and $\wpn(H)-1$ we let $X_j$ be $Y_j$.
Finally, we let $X_{\wpn(H)}=Y_{\wpn(H)} \cup Y_{\wpn(H)+1}$. 
But now for all $i$ between 1 and $\wpn(H)$, the graph $H[X_j]$ is in ${\cal F}_j$ 
(note that $H[X_{\wpn(H)}]$ is bipartite and so belongs to $\mathcal F_{\wpn(H)}$).
This contradicts the fact that ${\cal S}$ 
is an $H$-freeness  witnessing sequence.

We obtain similar contradictions if some $J_i$ contains no complement of a bipartite graph or no split graph.
In the first case, we apply the argument above with everything complemented. In the second case, we partition $H$
into $c+1$ cliques and $s-1$ stable sets, and let $X_{\wpn(H)}$ be formed by taking the vertices of one clique and one stable set.
\end{proof}

For many $H$ the following holds:  for almost every  graph $G$   in $\Forb_H^n$, there is a witnessing partition of $V_n$  certified 
by a  really canoncial witnessing $\wpn(H)$-sequence ${\cal S}=({\cal F}_1,\dots,{\cal F}_k)$. 

The results discussed above all show that this characterization hold for specific $H$.  
Combined with the results in this paper,   they  show the characterization  holds   whenever $H$ is a cycle or the 
complement of a cycle. In subsequent work, Reed and Yuditsky have  verified the charcterization whenever $H$ is a tree or the complement of the tree. 

A review of     Promel and Steger's 
 bounds on the number of  graphs in $\Forb_H^n$ for arbitrary $H$,  as a function  of $\wpn(H)$ and $n$ is useful for understanding the form the 
 characterization takes. 
 
Their result supports our decision to  focus on  partitions with $\wpn(H)$ sets and a witnessing sequence such that  each  ${\cal F}_i$  contains arbitrarily large graphs.  

To obtain a lower bound on $|\Forb_H^n|$ involving $\wpn(H)$, Promel and Steger  considered a fixed partition  $S_1,\ldots,S_{\wpn(H)}$  of $V_n$ such that for each $i,$ we have $\lfloor\frac{n}{\wpn(H)} \rfloor \le |S_i| \le \lceil \frac{n}{\wpn(H)} \rceil$. Clearly, every graph $G$ with vertex set $V_n$ in which $S_i$ is a clique  if $i$ is at most $c_H$ and stable otherwise is in $\Forb_H^n$. Since  we are free to choose which edges between the partition elements are in  $G$, it follows that $\Forb_H^n$ has size at least $2^{(1-\frac{1}{wpn(H)}){n \choose 2}}$. 

To obtain an upper bound on $|\Forb_H^n|$ involving  $\wpn(H)$, Promel and Steger (implicitly) showed that for every $H$ we can define a family of graphs $J_H$ such that  letting $J_H^n$ be the family of graphs in $J_H$ 
on $V_n$ we have: (i) the number of graphs in $J_H^n$ is $2^{o({n \choose 2})}$, and  (ii)  for almost every graph $G$  in $\Forb_H^n$, there is a partition $S_1,\dots,S_{\wpn(H)}$ of $V_n$ such that for every $i$, the subgraph $G[S_i]$ of $G$ induced by $S_i$ is
in $J_H$.   Now,  the number of  partitions  of $V_n$ into $\wpn(H)$ sets is 
at most $\wpn(H)^n$ which is $2^{o(n^2)}$. So,  provided $\wpn(H)$ is at least 2,  summing over all possible partitions and choices for the $G[S_i]$ yields that $|\Forb_H^n|=2^{(1+o(1))(1-\frac{1}{\wpn(H)}){n \choose 2}}$.

Combining the  upper bound   and the lower bound with a very short calculation actually allows us to say something about the size of the $S_i$ in the Promel-Steger partition of typical members of $\Forb_H^n$. For any particular partition $S_1,\dots,S_{\wpn(H)}$ of $V_n$ the number of graphs $G$ such that $G[S_i]$ is in  $J_H$ is by definition $2^{o({n \choose 2})}2^m$ where $m$ is the number of pairs of vertices of $V_n$ lying in distinct
elements of the partition. Setting $a_i=|S_i-\frac{n}{\wpn(H)}|$ and $d$ to be the sum of the squares of the $a_i$, 
we see that $m$ is essentially $(1-\frac{1}{\wpn(H)}){n\choose 2}-2d$. It follows from our lower bound on the size of $\Forb_H^n$ that  
we can strengthen the result of Promel and Steger ensuring a partition of almost every $H$-free graph mentioned in the last  paragraph by adding, 
for any $\epsilon>0$, the condition that every $S_i$ satisfies $\left| |S_i|-\frac{n}{\wpn(H)}\right| \le \epsilon n$.
This result  suggests that in characterizing $H$-free graphs we should be interested in partitions into $\wpn(H)$ parts each of which has about the same size. Hence, we should focus on really canonical $H$-freeness witnessing $wpn(H)$-sequences.

We hope this discussion provides some basis  for the precise form our characterization  takes. We now discuss those $H$ for which it holds. 

A cycle of length three    can be partitioned into a clique, and into three stable sets but not into two stable sets.
Thus $\wpn(C_3)$ is 2 and the set consisting of two copies of the family of all stable sets is a
$C_3$-freeness witnessing sequence. Thus, the  result of Erdos, Kleitman, and Rothschild mentioned above 
tells us that  our conjecture holds when $H$ is $C_3$. 

For $l$ at least four, the   cycle of length  $2l+1$  can be partitioned into three stable sets, into two stable sets  and a clique,
into a stable set of size $3$ and $l-1$ cliques, and into $l+1$ cliques. However, it cannot be partitioned 
into $l$ cliques. So, $\wpn(C_{2l+1})$ is $l$ and  the multiset consisting of $l$ copies of the family of all cliques is a
$C_{2l+1}$-freeness witnessing sequence. Thus, a  result of Balogh and Butterfield\cite{BB11}   mentioned above 
tells us that  our characterization holds when $H$ is $C_{2l+1}$. 

The   cycle of length $4$  can be partitioned into two stable sets  and into two cliques, However, it cannot be partitioned 
into a clique and a stable set. So, $\wpn(C_4)$ is $2$ and  the family of all cliques together with the family of all stable sets form a
$C_4$-freeness witnessing sequence. Thus  a result of Pr\"omel and Steger  mentioned above 
tells us that  our characterization  holds when $H$ is $C_4$. 

The   cycle of length $5$  can be partitioned into three stable sets, into three cliques, into two stable sets and a clique, and into two cliques and a stable set, However, it cannot be partitioned  into two  stable sets. 
So, $\wpn(C_5)$ is $2$. Furthermore, $C_5$ cannot be partitioned into a clique and a graph 
containing no $P_3$. Thus  the multiset consisting  of the  family of all cliques  together with the family of all disjoint unions of 
cliques  form a $C_5$-freeness witnessing sequence. Since $C_5$ is self-complementary so does the set consisting of  the family of all stable sets and
the family of all stable sets and complete multipartite graphs.   Thus  a result of Promel and Steger  mentioned above 
tells us that  our characterization  holds when $H$ is $C_5$. 

The   cycle of length $7$  can be partitioned into three stable sets,into four cliques,   into a stable set and three cliques, and into a clique and two stables sets.  
However, it cannot be partitioned  into three  cliques or into two cliques and a stable set. So, $\wpn(C_7)$ is $3$
and a result of Balogh and Butterfield mentioned above  shows that our characterization  holds when $H$ is $C_7$. 

The   cycle of length $6$  can be partitioned into two stable sets, into three cliques, and  into a stable set and two cliques.  
However, it cannot be partitioned  into two  cliques. So, $\wpn(C_6)$ is $2$. Furthermore, $C_6$ cannot be partitioned into a stable set and the complement of a graph of girth 5.
Thus  the
following theorem, proven in Section \ref{C6}, shows that our characterization holds  holds when $H$ is $C_6$. 

\begin{theorem}
\label{C6.thm}
Almost every $C_6$-free graph can be partitioned into a stable set and the complement of a graph of girth 5.
\end{theorem}

For $l>5$, the   cycle of length $2l$  can be partitioned into two stable sets, into $l$ cliques,  and into a stable set of size $4$  and $l-2$ cliques,  However, it cannot be partitioned  into $l-1$ cliques. 
So $\wpn(C_{2l})$ is $l-1$. Furthermore, $C_{2l}$ cannot be partitioned into $l-2$ cliques and a graph which has at most 3 vertices or has a 
disconnected  complement or contains $C_3$. So the following theorem, proven in Section \ref{C2l}, shows that our characterization  holds when $H$ is $C_{2l}$. 

\begin{theorem}
\label{C2l.thm}
For $l>5$,
almost every $C_{2l}$-free graph can be partitioned into $l-2$ cliques and the complement of a graph  which is the disjoint union of stars and triangles.
\end{theorem}

We note that this theorem was proven independendly  using different techniques by Kim et al. in \cite{KKOT}\footnote{The results in this paper 
were proven  and a draft written concurrently  but the paper lay dormant for many years}

The   cycle of length $8$  can be partitioned into two stable sets, into four cliques, and  into a stable set and three cliques.  
However, it cannot be partitioned  into three  cliques. So, $\wpn(C_8)$ is $3$. Furthermore, $C_8$ cannot be partitioned into two cliques,
and a  graph which  either has at most one edge, contains $C_3$, or  is  disconnected in the complement.
Thus  the following theorem proven in Section \ref{C2l} shows that our characterization  holds when $H$ is $C_8$.

\begin{theorem}
\label{C8.thm}
Almost every $C_8$-free graph can be partitioned into $2$ cliques and  a graph whose complement   is the disjoint union 
of graphs each of which  is  the join\footnote{The join of $G$ and $H$ has vertex set $V(H) \cup V(G)$ and edge set 
$E(H) \cup E(G) \cup \{xy| x \in V(H), y \in V(G)\}$. }  of a clique and a stable set. 
\end{theorem}

The   cycle of length $10$  can be partitioned into two stable sets,into five cliques,  and  into a stable set and  four cliques,  However, it cannot be partitioned  into four   cliques. So $\wpn(C_{10})$ is $4$. 
Furthermore, $C_{10}$ cannot be partitioned into three cliques
and a  graph which  either is disconnected in the complement,  contains a $C_3$, has at most three vertices, or is stable.  
Thus  the following theorem proven in Section \ref{C2l} shows that our characterization  holds when $H$ is $C_{10}$. 

\begin{theorem}
\label{C10.thm}
Almost every $C_{10}$-free graph can be partitioned into $3$ cliques and a graph  which is the complement of the disjoint union of stars and cliques.
\end{theorem}

Thus, our characterization  holds whenever $H$ is a cycle. Reed and Yuditsky have shown that it also holds whenever $H$ is a tree, 
Since the complement of a hereditary family is a hereditary family and the $\overline{H}$-free graphs are precisely the complements 
of the $H$-free graphs, it also holds when $H$ is the complement of a cycle or the complement of a tree. 

In the remaining sections of the paper, we prove the four theorems we have just stated.
We end this section with some definitions.

We use $P_k$ for the $k$-vertex path and $E_k={\overline K}_k$ for the $k$-vertex graph with no edges.  It will be convenient to use $\log$ for $\log_2$ throughout the paper.

We say that a hereditary family of graphs is {\em restricted} if 
it does not contain all bipartite graphs, does not contain all complements of bipartite graphs and does not contain 
all split graphs.

\section{The Relevant Witnessing Sequences, Their Structure and Growth Rate}

We turn now to determining the relevant   witnessing sequences for cycles of length $2l$ for $l>2$, 
and prove some results about the growth rate of the families in these sequences and the structure of typical graphs within them.

Throughout we use the fact proven by Seinsche\cite{Se}  (and rediscovered many times, see   \cite{BLS}) that for every $P_4$-free $G$, one of $G$ or $\overline{G}$ is disconnected.
We also use the observation (in both graphs and their complements) that a graph contains no induced $P_3$ if and only if it 
is the disjoint union of cliques. 

\subsection{The Really Canonical  Witnessing Sequences}

\begin{claim}
Every really canonical   $C_6$-freeness witnessing sequence ${\cal F}_1,{\cal F}_2$.
satisfies one of the following for some $i$:
 \begin{enumerate}

\item Every graph  in ${\cal F}_i$ is the complement of a graph of girth  5 and every graph  in ${\cal F}_{3-i}$ is  a stable set, 

\item Every graph in  ${\cal F}_{i}$ is $P_4$-free and every graph in ${\cal F}_{3-i}$  of size at least three is a clique, 

\item Every graph in ${\cal F}_i$ is stable or complete multipartite, and every graph in ${\cal F}_{3-i}$ is a clique or  the complement of a star, 

\item Every graph   in  ${\cal F}_i$ is the complement of a matching, and every graph in  ${\cal F}_{3-i}$ is the disjoint union of a clique and a stable set.

 \end{enumerate}
\end{claim}

\begin{proof}
If for some $i \in \{1,2\}$, ${\cal F}_i$ is the family of all stable sets  then ${\cal F}_{3-i}$ 
contains neither $2K_2$ nor $S_3$ and hence every graph in this family is the complement of a graph of girth 5.

Otherwise  both ${\cal F}_1$ and  ${\cal F}_2$  contain the clique of size 2 and hence both are  subfamilies of the family of $P_4$-free graphs.
If either ${\cal F}_1$ or ${\cal F}_2$ consists only of  cliques and the stable set of size 2  we simpy record that there is an $i$ in $\{1,2\}$ such that 
 every graph in $F_i$ is $P_4$-free,  and every graph in ${\cal F}_{3-i}$ is a clique or the stable set of size 2. 

Otherwise  ${\cal F}_1$ contains a stable set  of size $2$  
so ${\cal F}_2$  also contains no $2K_2$, and similarly ${\cal F}_1$ contains no $2K_2$ . 
Furthermore, each of ${\cal F}_1$ and ${\cal F}_2$ contain at least  one of $P_3$ or $\overline{P_3}$.
Since a $C_6$ can be partitioned both into two $P_3s$ and into two $\overline{P_3}$s we see that for some 
$i$, we have that ${\cal F}_i$ contains $\overline{P_3}$ but not $P_3$ while 
${\cal F}_{3-i}$ contains $P_3$ but not $\overline{P_3}$. 
In this case every graph in ${\cal F}_{3-i}$ is the disjoint union of a clique and a stable set
while every graph in ${\cal F}_i$ is stable or a complete multipartite graph.
Since $C_6$ can be partitioned into two stable sets of size 3 we see that in fact either ${\cal F}_i$ contains only 
complements of  matchings or  ${\cal F}_{3-i}$ contains only cliques and complements of stars. 

\end{proof} 

\begin{claim}
Every really canonical   $C_8$-freeness witnessing sequence
${\cal F}_1,{\cal F}_2,{\cal F}_3$ satisfies one of the following: 

 \begin{enumerate}
  \item  one family contains only graphs $J$ such that 
  every component of $\overline J$ is the join of a clique and a stable set (and hence induces the disjoint union of a clique and a stable set in $J$);
  the other two families  contain only cliques and possibly the stable set of size 2
  \item  one family contains only cliques and stable sets, a second family   contains only cliques and possibly the stable set of size 2, and the third 
  contains only complements of matchings.
 \end{enumerate}

\end{claim}

\begin{proof}

If there is no $\mathcal F_i$ which contains a $\overline{P_3}$, then each family contains only stable sets, cliques,  and 
complete multipartite  graphs. Furthermore at most one family contains an $E_3$ and at most one family contains a $P_3$,
so  either two families contain only  cliques and possibly the stable set of size 2, or one family contains only cliques and possibly the stable set of size 2,  
while  a  second contains  only  complements of  matchings and a third contains only cliques and stable sets.  These options are both covered in the statement of the claim.

Otherwise,  there is a unique $i$ which contains a $\overline{P_3}$. This implies that the   other ${\cal F}_i$ cannot contain 
$E_3$ or $P_3$ so they  consists of all cliques and possibly the stable set of size 2.  Hence ${\cal F}_i$  contains no $P_4$,  and no  graph with 4 vertices and 2 edges.  
So, for every graph $J$ in $F_i$ each of the components of $\overline{J}$  induces a  disconnected subgraph of $J$, and these subgraphs have no nonclique components 
and at most one  component which is not a vertex. 

\end{proof}

\begin{claim} 
Every really canonical   $C_{10}$-freeness witnessing sequence
${\cal F}_1,{\cal F}_2,{\cal F}_3,{\cal F}_4$ satisfies:

one  family  consists of graphs which are joins of graphs which are either 
stable sets or the disjoint union of a clique and a vertex, and every other family  consists of all cliques and possibly $E_2$. 
\end{claim}

\begin{proof}
$C_{10}$ can be partitioned into  an $L$, an $M$,  and   $a$ edges and $b$ non-edges for any $L,M \in\{P_3,\overline{P_3}, E_3\}$ and
$a$ and $b$ which sum to two. 
So there is at most one ${\cal F}_i$ such that $J_i$ does not contain $\{P_3,\overline{P_3}, E_3\}$, and every 
 other family consists of all cliques and possibly the stable sets of size 2.
  Since $C_{10}$ can be partitioned into a  $P_4$ and 3  edges,
we see that  ${\cal F}_i$ contains no  $P_4$. 
A similar argument shows that ${\cal F}_i$ contains no graph with 4 vertices and either one or two  edges. 
Since every $P_4$-free graph is either disconnected or disconnected in the complement,   it follows that   every graph in ${\cal F}_i$ is the complement of the disjoint union of stars and cliques. 
\end{proof}

.

\begin{claim}\label{c10}  For $l>5$, $C_{2l}$ can be partitioned into any of the following sets of graphs:

(a) An $L$, an $M$,  and $a$ edges and $b$ non-edges for any $L$ and $M$ which are both  in $\{P_3,\overline{P_3}, E_3\}$ and any $a$, $b$ summing to $l-3$.

(b) A $P_4$, $a$ edges and $b$ non-edges for any $a+b=l-2$. 

(c) An $H$, $a$ edges and $b$ non-edges for any $a+b=l-2$, and any graph $H$ with 4 vertices and at most two edges.  
\end{claim}

We use Claim \ref{c10} to show the following.

\begin{claim}\label{c11}
For every  $l>5$, Every really canonical   $C_{2l}$-freeness witnessing sequence
${\cal F}_1,{\cal F}_2,...,{\cal F}_{l-1}$ satisfies:
There is an index $i$ such that  the family  ${\cal F}_i$ .  
every graph in the family  ${\cal F}_i$ is the complement of the disjoint union of stars and triangles
while every other family consists of all cliques or all cliques and the two vertex graph $E_2$
\end{claim}

Note that here we allow trivial (1-vertex) stars, so $\mathcal F_{i^*}$ may contain cliques.

\begin{proof}
It follows from Claim \ref{c10}(a) that there is at most one index $i$ such that $\mathcal F_i$ contains any graph
from $\{P_3,\overline{P_3}, E_3\}$.  So we can choose $i$ such that, for all $j\ne i$, the family  ${\cal F}_j$ consists of all cliques or all cliques and the two vertex graph with no edge.   

It follows from Claim \ref{c10}(b) that  ${\cal F}_i$ contains no  $P_4$, and from Claim \ref{c10}(c) that ${\cal F}_i$ contains no graph with 4 vertices and at most two edges.  
Since any $P_4$ free graph is either disconnected or disconnected in the complement, it follows that for any $J$ in ${\cal F_i}$ the components of $\overline{J}$
are either the union of a clique and a vertex, or a stable set of size  three. 
 
\end{proof}

\subsection{Growth Rate and Structure}

In order to prove our results we will need to bound the number of graphs certified as $H$-free by  a partition using one of the relevant witnessing sequences 
and compare this to the number which have no such witnessing partition.  Doing so  involves 
 some careful counting which will require  us to know how many graphs with a given number of 
 vertices there are in the various ${\cal F}_i$.  It will also help to know something about the typical behaviour of 
 the graphs which have witnessing partitions We now state some results which will be useful in this regard. 

Every $P_4$-free graph  is either disconnected or disconnected in the complement
 This, together with the fact that a graph cannot be both disconnected and disconnected in the complement 
 implies that there is a bijection between disconnected  $P_4$-free graphs  on $V_n$  
 and rooted trees which have $n$ leaves labelled with the elements of $V_n$ and no vertices 
 with exactly one child. If $G$ is a single vertex then the tree is a single node. Otherwise, 
 the  children of the root for the  tree for $G$ are in $1-1$ correspondence with the components of $G$.
 The subtree consisting of the descendants of a child of the root corresponding to a component $U$ of $G$
 form the tree corresponding to the complement of $U$.  Of course, proceeding analogously, we get a bijection between 
 such trees and the connected $P_4$-free graphs on $V_n$. Since each such tree has at most $2n$ vertices,
and (by Cayley's formula)  there are $l^{l-2}$ trees on $l$ vertices, we obtain that the number of $P_4$-free graphs on $n$ vertices is less than  $(2n)^{2n}$.
 
 We need the following results  on graphs of girth 5 from Morris and Saxton \cite{MS} (the second is a weak version of their Corollary 1.3).
  
 \begin{theorem} 
 \label{MStheorem}
 There are $2^{\Theta(n^{3/2})}$  graphs  of girth 5 with vertex set $V_n$.
 \end{theorem} 
 
 \begin{theorem}
 \label{belasam}
 There is an $\epsilon>0$ such that the proportion of graphs of girth 
 5  on $V_n$ with fewer than $\epsilon n^\frac{3}{2}$ edges is $o(n^{-3n})$. 
 \end{theorem}
 
 Along with these results, we need the following claims:
 
 \begin{claim}
 \label{girth5hd.claim}
 Let $G$ be a graph of girth 5 with $n$ vertices.  Then $G$ contains at most $\sqrt{n}$ vertices of degree exceeding $\frac{3\sqrt{n}}{2}$
 and the sum of the degrees of these vertices is at most $\frac{3n}{2}$.
 \end{claim}
 
 \begin{proof}
Any two vertices can have at most one common neighbour.
This implies that any $i$ of them have total degree at most $n+{i \choose 2}$, which would lead to a contradiction if there 
were $\sqrt{n}$ vertices of degree exceeding $\frac{3\sqrt{n}}{2}$.   The second bound follows since $i\le\sqrt n$ implies that $n+{i \choose 2}\le 3n/2$.
\end{proof}

 \begin{claim}
 \label{girth5hd.claim2}
We let $s(G)$ be the maximum size of a subset $S$ of $V(G)$ such that $S$ is a stable set and there is no vertex of $G$ with edges to two vertices of $S$.
 The proportion of graphs of girth at least five on $V_n$ vertices for which $s(G)> 3n^{9/10}$ is $2^{-\omega(n)}$.
 \end{claim}
 
 \begin{proof}
 Suppose  some  set $S$ of $\lceil 3n^{9/10} \rceil $  vertices of $G$ is stable, and no 
 vertex has edges to two of its elements.  We note that this implies that the sum of the degrees of the vertices in $S$ 
 is less than $n$. Combined with Claim \ref{girth5hd.claim} it also implies that the number of vertices of $S$ adjacent to a vertex of degree 
 exceeding $\frac{3\sqrt{n}}{2}$ is at most $\sqrt{n}$. Thus, for sufficiently large $n$ there is a subset $S'$ of $S$ of size 
 $\lceil n^{9/10} \rceil $ every vertex of which has degree less than $n^{1/10}$ and has no neighbour of degree exceeding 
 $\frac{3\sqrt{n}}{2}$.  
 
 Theorem \ref{belasam} implies that to  complete the proof, we need only show that for  any $\epsilon>0$ and each subset $D$ of $V_n$ of  size  $\lceil n^{9/10} \rceil $  the proportion  
   of graphs of girth at least five  on $V_n$ which have  at least $\epsilon n^{3/2}$ edges in which $D$  can play the role of $S'$ is $2^{-\omega(n)}$
 
  In order to do so,  for every $j<n^{21/20}$, we consider  the family ${\cal D}_j$ of those $G$ satisfying the following:
  \begin{enumerate}
  \item[(i)] $G$ has  girth at least five and more than $\epsilon n^{3/2}$ edges,
  \item[(ii)] no vertex of $D$ has a neighbour of degree exceeding $\frac{3\sqrt{n}}{2}$.  
  \item[(iii)] no vertex outside $D$ sees two vertices of $D$, 
 \item[(iv)] $D$   spans $j$ edges, the sum of the degrees of the vertices of $D_j$ is at most $n+2j$ and  no vertex of $D$ has degree exceeding $n^{1/4}$. 
 \end{enumerate}
  
 For each $j$ between $1$ and $n^{21/20}$,  we consider an auxiliary bipartite graph  $B_j$ where one side consists of the elements of ${\cal D}_j$ and the other consists 
 of the elements of ${\cal D}_{j-1}$. Two graphs are joined  by an edge in  $B_j$ if the graph in ${\cal D}_j$ can be obtained from the graph in ${\cal D}_{j-1}$
 by adding an edge with endpoints in  $D$ and deleting an edge with at least one endpoint outside $D$.
 
 Each graph in ${\cal D}_j$ has degree at most  $jn^2<n^{31/10}$  in $B_j$.
 
  For any graph in ${\cal D}_{j-1}$  to obtain a neighbour in $B_j$ we can delete any of the at least 
 $\epsilon  n^{3/2}-j >\frac{\epsilon n^{3/2}}{2}$ of its edges  whose endpoints are not both in $D$ and add any edge within $D$ between vertices each of which have degree lesss than  $n^{1/4}-1$
 which are not joined by a path with at most $3$ edges. Since the sum of the degrees of the  $\lceil n^{9/10} \rceil $ vertices of $D$ is at most $n+2j<3n^{21/20}$,
 more than half the vertices of $D$ have degree less than  $n^{1/4}-1$. For any such vertex $z$, there are less than $2n^{3/4}$ vertices of $D$ which. are 
 joined to $z$ by a path with at most 3 edges remaining in $D$. Furthermore, since every vertex outside $D$ has only one neighbour in $D$, and every 
 neighbour of $z$ has at most $\frac{3\sqrt{n}}{2}$ neighbours, there are less than $\frac{3n^{3/4}}{2}$ vertices of $D$ which. are 
 joined to $z$ by a path with at most 3 edges which leaves $D$. It follows that the number of $v$ in  $D$ for which adding the  edge $zv$ 
 and deleting any edge not within $D$ yileds a graph in ${\cal D}_j$ is at least $\frac{|D|}{3}$. So 
 every  graph in ${\cal D}_{j-1}$ has degree at least $\frac{|D|^2\epsilon n^{3/2}}{12}>n^{32/10}$ in $B_j$. 
 
 Counting the edges of $B_j$ from both sides we obtain: $|{\cal D}_{j-1}|<n^{1/10}|{\cal D}_j|$ and hence $ |{\cal D}_0|<n^{\frac{n^{21/20}}{10}}|{\cal D}_{n^{21/20}}|$.
 
\end{proof}

 \begin{claim}
 \label{girth5hd.claim3}
 The proportion of graphs of girth at least five on $V_n$  with maximum degree exceeding $\frac{n}{\log \log n}$ is $2^{-\omega(n)}$.
 \end{claim}
 
 \begin{proof}

 It is enough to show for each vertex $v$  in $V_n$, that the proportion of graphs of girth at least five on $V_n$  in which $v$ has degree exceeding  $\frac{n}{\log \log n}$ is $2^{-\omega(n)}$.
  
 For each even  $j$ between $\lceil \frac{n}{2 \log \log n} \rceil$ and $\lceil \frac{n}{ \log \log n} \rceil$,  we let ${\cal D}_j$  be the graphs of girth five  on $V_n$  in which 
 $v$ has degree $j$. We  consider an auxiliary bipartite graph  $B_j$ where one side consists of the elements of ${\cal D}_j$ and the other consists 
 of the elements of ${\cal D}_{j-2}$. Two graphs are joined  by an edge in  $B_j$ if the graph in ${\cal D}_j$ can be obtained from the graph in ${\cal D}_{j-2}$
 by deleting some edge and adding edges from  its  endpoints to  $v$. 
 
 Now, Claim \ref{girth5hd.claim} implies that a graph of girth 5 has fewer than $3n^{3/2}$ edges so each graph in ${\cal D}_{j-2}$ has degree at most  $3n^{3/2}$  in $B_j$.
 
  For any graph in ${\cal D}_{j}$  to obtain a neighbour in $B_j$ we can delete the edges to any pair of neighbours of $v$  which are not joined by a path of $G-v$  with at most three edges. 

  Since $G$ has girth 5, $N(v)$ is stable, and each vertex has at most one neighbour in $N(v)$. So, there are at most $\sqrt{n}$ vertices of $N(v)$ joined to a vertex  of degree exceeding $\frac{3 \sqrt{n}}{2}$ and $n^{3/4}$ which have degree exceeding $n^{1/4}$.
   
Thus,   more than half the vertices of $N(v)$ have degree less than $n^{1/4}$ and have no neighbour   of degree exceeding $\frac{3 \sqrt{n}}{2}$. . For any such vertex $z$, since every vertex outside $N(v)$ has only one neighbour in $D$, and every 
 neighbour of $z$ has at most $\frac{3\sqrt{n}}{2}$ neighbours, there are less than $\frac{3n^{3/4}}{2}$ vertices of $N(v)$ which. are 
 joined to $z$ by a path with at most 3 edges. It follows that the number of $y$ in  $N(v)$ for which deleting  $vz$ and $vy$ and adding the edge 
 $yz$ yileds a graph in ${\cal D}_{j-2}$ is at least $\frac{|N(v)|}{3}$. So, 
 every  graph in ${\cal D}_{j}$ has degree at least $\frac{j^2}{6}>3n^{7/4}$ in $B_j$. 
 
 Counting the edges of $B_j$ from both sides we obtain: $|{\cal D}_{j}|<n^{-1/4}|{\cal D}_{j-2}|$ and hence $ |{\cal D}_\frac{n}{\log \log n}|<n^{\frac{-n}{10 \log \log n}}|{\cal D}_\frac{n}{2\log \log n}|$.

\end{proof}

It will be useful to have estimates on the size of several families.  Let $\mathcal G_n$ be the set of graphs with vertex set $V_n$, and define
\begin{align*}
{\mathcal F}^*_1(n)&=\{G\in\mathcal G_n:\mbox{$\overline G$ is a disjoint union of stars and triangles}\}\\ 
{\mathcal F}^*_2(n)&=\{G\in\mathcal G_n:\mbox{$\overline G$  is a disjoint union of stars and cliques}\}\\ 
{\mathcal F}^*_3(n)&=\{G\in\mathcal G_n:\mbox{every component of $\overline G$ is the join of a clique and a stable set}\}.
\end{align*}
We include trivial stars (so $K_1$ is allowed) and empty cliques and stable sets in these definitions.  We also define, for $i=1,2,3$,
$f^*_i(n)=|\mathcal F^*_i(n)|$.  Finally, let $B_n$ be the $n$th Bell number (the number of partitions of $[n]$).

\begin{lemma}\label{finsize}
 For $i=1,2,3$, $f^*_i(n)$ is monotonic in $n$ and $B_n\le f^*_i(n)\le 2^nB_n$.
\end{lemma}

\begin{proof} We work with $\overline G$ rather than $G$ in each case, so our graphs are disjoint unions of (i)stars and triangles, (ii) stars and cliques, or (ii) 
joins of a clique and a stable set. The lower bound is trivial in each case, as for every
partition of $V_n$ we can place a star or clique component into each set.

We can specify the complement of a graph in $\mathcal F^*_1(n)$ or $\mathcal F^*_2(n)$ by giving the partition of $V_n$ into vertex sets of components and the set of vertices with degree at least 2.
Similarly, we can specify the complement of a graph in $\mathcal F^*_3(n)$ by giving the partition of $V_n$ into vertex sets of components 
and the set of clique vertices.  This gives the upper bound.

Finally, monotonicity follows by considering elements of $\mathcal F^*_i(n)$ for which $n$ is a universal  vertex.
\end{proof}

\begin{corollary}
For every $l>3$, the number of $C_{2l}$ free graphs is at least $2^{(1-\frac{1}{l-1}){n \choose 2}} B_{\lceil \frac{n}{l-1} \rceil}$. 
\end{corollary}

We need some results on the asymptotics of these functions
 
\begin{observation}
\label{obs:bruce4} 
Setting $l=l(n)=l=\lceil \frac{n}{\ln n} \rceil$,   for $n \ge 8$ we have $f_i(n)>B_n \ge \frac{l^n}{2(l!)}> l^{n-l}>({l}/{2})^n$. 
Moreover,  as shown by  de Bruijn \cite{deB}, $B_n=\frac{l^n}{l!} 2^{(1+o(1))n}$. 
Finally, the  proportion  of partitions  of $[n]$ which have fewer than $\frac{l}{ \alpha}$ nonsingeleton parts 
 is less than  $(\frac{\alpha}{4})^{-n}$. 
\end{observation}
\begin{proof}
For $n \ge 8$, if we partition $V_n$ into $l=\lceil \frac{n}{\ln n} \rceil$ sets 
by randomly placing each vertex into a set then the expected number of empty sets 
is $l(1-\frac{1}{l})^n$, which is less $\frac{1}{2}$. So 
$f_i(n)>B_n \ge \frac{l^n}{2(l!)}> l^{n-l}>({l}/{2})^n$.
Now the number of partitions of $V_n$ into  a set of singletons and  at most ${l}/{\alpha}$ 
parts is less than $2^n({l}/{\alpha})^n<(l/2)^n\cdot (\frac{\alpha}{4})^{-n}$. 
\end{proof}

\begin{observation}
\label{obs:bruce5}
The proportion  of partitions  of $N$ which have  more than $\frac{n}{\sqrt{\ln n}}$ 
parts is    is  $o(2^{-n(\ln n)^{1/3}})$. 
\end{observation}
\begin{proof}
The number of partitions of $N$  which have more than $k$ empty  parts is at most $\frac{n^n}{k!}$.
By Observation \ref{obs:bruce4}, $2B_n \ge( \frac{n}{\ln n})^n\frac{1}{(n/\ln n)!}$. 
 The result follows. 
\end{proof}

\begin{observation}
\label{obs:bruce9}
The proportion  of partitions  of $N$ which have  more than $\frac{n}{\sqrt{\ln \ln n}}$ 
vertices in parts of size greater than $(\ln n)^3$    is  $o(2^{-n(\ln\ln n)^{1/3}})$. 
\end{observation}
\begin{proof}
We can specify the partition of the elements of  $N$ into those in parts of size  exceeding  $(\ln n)^3$ 
and those in smaller parts in $2^n$ ways. Letting $n_1$ be the number of elements in larger parts,
and $k$ be the number of larger parts,  we can specify the larger parts in at most 
 in $k^{n_1}/k!$ ways. We note $k<\frac{n_1}{(\ln n)^3}$. 
We can specify the remaining parts in $B_{n-n_1}$ ways. 
By Observation \ref{obs:bruce5}, $B_j = 2^{(1+o(1))j} (\frac{j}{\ln j})^j \frac{1}{(j/\ln j)!}$. 
 The result follows. 
\end{proof}

This gives bounds on the asymptotics of the functions $f^*_i(n)$; however, we will also need a local bound on their growth.

\begin{observation}
\label{obs:thesizeofF_i}
For $n\ge8$ and $i=1,2,3$, we have $\frac{nf^*_i(n)}{16\log n}\le f^*_i(n+1)\le 4nf^*_i(n)$. 
\end{observation}

\begin{proof}
Since every element of $\mathcal F^*_i(n)$ is determined by a partition and a subset of $[n]$, it 
follows  from an application of Observation \ref{obs:bruce4} with $\alpha=8$, that more than half the graphs
in ${\mathcal F}^*_i(n)$ have at least $\frac{l}{8}$ components.

Given a  component $K$ of a graph in ${\mathcal F}^*_i(n)$, we can obtain a graph in ${\mathcal F}^*_i(n+1)$
in which $v_{n+1}$ is not universal  by enlarging $K$ as follows.   
For $i=3$, we add $v_{n+1}$ to $K$ as a universal vertex. 
For $i=2$, if $K$ is a clique we add $v_{n+1}$ as a universal vertex; otherwise we make it adjacent only to the center of the star $K$. 
For $i=1$, if $K$ is a star we add $v_{n+1}$ as a leaf; otherwise, we add $v_{n+1}$ as the center of a star whose 3 leaves are the vertices of $K$.  
Each graph  in $\mathcal F^*_i(n+1)$ in which $v_{n+1}$ is not universal  arises 
from a unique choice of  a graph  $G$ in ${\mathcal F}^*_i(n)$ and a component $K$ 
of $\overline{G}$.  There are of course $f^*_i(n)$ graphs in ${\mathcal F}^*_i(n+1)$ for which  
$v_{n+1}$ is universal. 
Hence, by the discussion in the previous paragraph, we have
$f^*_i(n+1)>(1/2)f^*_i(n)\cdot (l/8)\ge f^*_i(n)\cdot n/16\log n$.

Finally, for the upper bound, we show that $f^*_i(n+1)\le 3(n+1)f^*_i(n)$.
We construct a mapping from $\mathcal F^*_i(n+1)$ to $\mathcal F^*_i(n)$ as follows.  
Given a graph $G$ in $\mathcal F^*_i(n+1)$, let $K$ be the component of $\overline G$ that contains $1$.  
If $K$ is just the vertex 1 then delete 1.  If $K$ is a clique, or a join of a clique and a stable set, 
then delete the vertex of $K$ with the largest label.  If $K$ is a star then delete the leaf
with the largest label.  Finally, order the  vertices so their labels are increasing and relabel  with labels $1,\dots,n$  so this order remains the same.   We note that to reconstruct the original graph from the image if we know the label of the vertex that was deleted, we simply need to determine the  adjacency of the deleted vertex.
If this label was 1 then $v_1$ is universal and we are done. Otherwise, 
when $i=3$, this is determined by specifiying whether the deleted vertex came from the  clique or the stable set, 2 choices. 
Otherwise if the component contained at least four vertices before deletion we knew whether it was a star or a clique, and there is a unique 
choice for the neighbourhood of the deleted vertex. If the component had three vertices there are 3 possible neighbourhoods for the deleted vertex (it must see at least one of the 
other vertices). 
  Thus $f^*_i(n+1)\le 3(n+1)f^*_i(n)$.
\end{proof}

\section{A First Near Witnessing Partition}
\label{prelim.sec}

The first step in our proof is to show that almost every $H$-free graph permits an $H$-freeness witnessing partition of all but $o(n)$ 
of its vertices. A  beautiful paper of Alon, Balogh, Bollobas, and Morris \cite{ABBM}  allowed us to simplify this part of the  proof significantly. 

Their discussion is in terms of partitions such that no part contain a specific bipartite graph $U(k)$, defined as follows: for
$k\ge1$, the graph $U(k)$ has bipartition $(A,B)$ with $|A|=k$ and $|B|=2^k$, and for every subset $S$ of $A$ there is
a vertex $b\in B$ with $\Gamma(b)=S$.  Abusing notation slightly, we shall say that a graph $G$ is {\em $U(k)$-free} if
there is no pair $S_T$ of disjoint sets of vertices of $G$ such that the bipartite subgraph of $G$ between $S$ and $T$ is a copy of $U(k)$ (in other words, if we ignore the edges inside $S$ and $T$ we get an induced copy of $U(k)$).
As Alon et al.   prove using two applications of the duality of VC-dimension, for any 
restricted family ${\cal F}$ there is a $k$ such that no graph in ${\cal F}$ contains $U(k)$ as a subgraph
(this is their Lemma 7 for $r=1$). 

So, their Corollary 8, which bounds the number of graphs of size $l$ which are $U(k)$-free     immediately yields the following strengthening. 

\begin{corollary}
\label{restrictbound.cor}
For every restricted hereditary family   ${\cal F}$  there is a positive  $\epsilon$ such that for every sufficiently large $l$, the number of graphs in ${\cal F}$ with $l$ vertices 
is less than $2^{l^{2-\epsilon}}$. 
\end{corollary}

The following is  essentially an immediate corollary of their Theorem 1.

\begin{corollary}
\label{abbm.cor}
For every $H$ with at least two vertices,  every sufficiently large $k$ and every $\delta>0$, there are  positive  $\epsilon$ and $b$   such that the following holds:

For almost every $H$-free graph   $G$ on $V_n$  there is a  set $B$ of at most $b$ vertices and a partition of $V_n$ into  $S_1,\dots,S_{\wpn(H)}, A_1,\dots,A_{\wpn(H)}$ and a set $B$ such that:
\begin{enumerate}
\item[(a)] $G[S_i]$ is $U(k)$-free for each $i$; 
\item[(b)] $|A_1 \cup\dots \cup A_{\wpn(H)} | \le n^{1-\epsilon}$; 
\item[(c)] for every vertex $v\in S_i \cup A_i$ there is a vertex $a\in B$ such that 
$$|(N(v)\triangle N(a))\cap (S_i\cup A_i)|\le \delta n;$$
\item[(d)] for every $i$, we have 
$$\left| |S_i \cup A_i|-\frac{n}{\wpn(H)}\right| \le \delta n.$$ 
\end{enumerate} 
\end{corollary}

\begin{proof}
We only need to prove this result for $\delta$ sufficiently small as it then follows for all $\delta$.
We can essentially read this out of the proof of Theorem 1 in 
Alon, Balogh, Bollobas, and Morris \cite{ABBM}, letting the hereditary family ${\cal P}$ be $\Forb_H$ 
and $\alpha={\delta}/{3}$ after choosing $\delta$ small enough that $\alpha$ is as small as required by their proof. We omit the details, simply sketching the very. very minor modifications. 

 We note that their $\chi_c({\cal P})$ is exactly $\wpn(H)$. We want  to use the  strengthening  of their Lemma 23
obtained by (i) replacing $\alpha=\alpha(k,{\cal P})>0$ in its statement with $\alpha>0$ sufficiently small in terms of $k$ and ${\cal P}$, and (ii) replacing {\it with} in the definition of $U({\cal P}_n,\alpha,k)$ just before lemma 23 by {\it with $|B|>c(\alpha,{\cal P})$ for the $c$ of Lemma 19 or}.  Their proof of the lemma actually proves this strengthening provided that in the definition of $U_n$ given on its fourth line we replace $n^{1-\alpha}$ by $c$.  

Now while  following their (three paragraph) proof of their Theorem 1, we again replace $\alpha=\alpha(k, {\cal P})$ by $\alpha>0$ sufficiently small in terms of $k$ and ${\cal P}$. We also add
$and |B| \le c$ at the end of the second paragraph before {\it for almost all}.  

Then we consider the adjustment $S^\prime_1,\dots,S^\prime_r$ and exceptional set $A$ they obtain  and set  $A_i=S^\prime_i \cap A, S_i=S_i^\prime-A$. Now, as in their proof,  consider a maximal $2\alpha$ bad set 
$B$. By our strengthened version of their Lemma 23,  the size of $B$ is at most some constant $c$. We set $b$ to be this $c$. 
Now, (a) is their Theorem 1(b); (b) is their Theorem 1(a) where $\epsilon =\frac{\alpha}{2}$; (c) follows immediately from the fact that 
$S^\prime_1,\dots,S^\prime_{\wpn(H)}$ is an $\alpha$-adjustment and the definition of an adjustment;
and (d) follows from their initial choice of partition (at the beginning of their proof) and the definition of an adjustment.
\end{proof}

We will need a strengthening of this theorem for $H$-free graphs.  In our strengthening we use 
$X_i$ in place  of $S_i$ and $Z_i$ in place of $A_i$ to avoid confusion. We will need a sharper bound than above on the
size of each set in the partition; and we will want to insist that, for some really canonical
$H$-free witnessing $\wpn(H)$-sequence $\mathcal F_1,\dots,\mathcal F_{\wpn(H)}$, $G[X_i]$ is contained in  ${\cal F}_i$ for each $i$.

\begin{theorem}
\label{Hstruc.thm}
For every $H$ with $wpn(H) \ge 2 $,  and  $\delta>0$ sufficiently small in terms of $H$, there are  $\gamma>0$
and an integer $b$ such that the following holds:

For almost every $H$-free graph   $G$ on $V_n$  there is a partition of $V_n$ into  $X_1,\dots,X_{\wpn(H)}, Z_1,\dots,Z_{\wpn(H)}$ such that for some set $B$ of at most $b$ vertices    each $G[X_i]$ is $U_k$-free and the following hold:

\begin{enumerate}
\item[(I)] There is a really canonical $H$-freeness witnessing $\wpn(H)$-sequence
\goodbreak $\mathcal F_1,\dots,\mathcal F_{\wpn(H)}$
certifying the partition $X_1,\dots,X_{\wpn(H)}$ of $G[X_1 \cup\dots\cup X_{\wpn(H)}]$;
\item[(II)] $|Z_1 \cup Z_2 \cup\dots \cup Z_{\wpn(H)} | \le n^{1-\gamma}$;
\item[(III)] for every vertex $v$ of $X_i \cup Z_i$ there is a vertex $a$  of $B$ such that 
$$
|(N(v)\triangle N(a)) \cap (X_i \cup Z_i)|<\delta n;
$$
\item[(IV)] for every $i$, we have
$$\left| |Z_i \cup X_i|-\frac{n}{\wpn(H)}\right| \le n^{1-\frac{\gamma}{4}};$$ 
\end{enumerate}
\end{theorem}

We remark that if $wpn(H)=1$ then $H$ contains no cycle, and since $wpn(H)=wpn(\overline{H})$ neither does its complement. I.e.  $H$ is in $P_4,P_3,\overline{P_3},E_2,P_2,P_1$. The structure of $H$-free graphs for these
$H$ are well understood. 

\begin{proof}
We choose $k$ sufficiently large and then $\delta<\frac{1}{10\wpn(H)}$ sufficiently small in terms of $H$ and $k$. 
We note that if either Corollary \ref{restrictbound.cor} or   Corollary \ref{abbm.cor} holds for a specific choice of $\epsilon$ it also holds for all smaller $\epsilon$.
So we can choose $\epsilon,b>0$ such that  both these  corollaries  hold for these choices of $k,\delta,b$, and $\epsilon$. 
 Wec set $\gamma= {\epsilon}/{10}$ and 
consider $n$ large enough to satisfy certain implicit inequalities below. We  know that for almost every graph in $\Forb_H^n$ there are sets $B$, $S_i$ and $A_i$ satisfying Corollary  \ref{abbm.cor} above.
This partition satisfies  (a), (b), (c): we now show that 
for almost every graph in 
$\Forb_H^n$ with such a partition, we can obtain a partition satisfying (I),(II),(III) and (IV).  

We first show that we may assume that $\left| |S_i \cup A_i|- \frac{n}{\wpn(H)} \right|$ is at most $n^{1-\gamma}$. 
There are $2^{O(n)}$ choices for the partition.
Since  $S_i$ is $U(k)$-free and $n$ is large, Corollary \ref{restrictbound.cor} tells us that there  are at most $2^{n^{2-\epsilon}}$ choices for $G[S_i]$. 
The number of  choices for the edges incident with $\cup_i A_i$ is at most $2^{n|\cup_i A_i|}\le 2^{n^{2-\epsilon}}$. It follows that there are 
at most $2^{O(n^{2-\epsilon})}$ choices 
for the partition $S_1,\dots,S_{\wpn(H)},A_1,\dots,A_{\wpn(H)}$ and the graphs $G[S_1 \cup A_1],\dots,G[S_{\wpn(H)} \cup A_{\wpn(H)}]$. 
Define $a_i$ by $|S_i\cup A_i|=n/\wpn(H)+a_i$, so $\sum_{i=1}^{\wpn(H)}a_i=0$.  Then the number of edges not yet determined is at most
\begin{align*}
\binom n2-\sum_{i=1}^{\wpn(H)}\binom{|S_i\cup A_i|}{2}
&=\binom n2-\sum_{i=1}^{\wpn(H)}\binom{n/\wpn(H)+a_i}{2}\\
&=\binom n2-\frac{n^2}{2\wpn(H)}-\frac12\sum_{i=1}^{\wpn(H)}a_i^2+O(n).
\end{align*}
If $|a_i|\ge n^{1-\epsilon/4}$ for any $i$ then this is at most
$\binom n2-\frac{n^2}{2\wpn(H)}-\Omega(n^{2-\epsilon/2})$.  Since there are only $2^{O(n^{2-\epsilon})}$ choices for
the partition and the graphs $G[S_i\cup A_i]$, the number of graphs for which there is a partition such that 
$|a_i|\ge n^{1-\epsilon/4}$ for some $i$ is at most 
$2^{\binom n2-\frac{n^2}{2\wpn(H)}-\Omega(n^{2-\epsilon/2})}2^{O(n^{2-\epsilon})}
=2^{\binom n2-\frac{n^2}{2\wpn(H)}-\Omega(n^{2-\epsilon/2})}$.  
Since $\Forb_H^n$ has size at least $2^{(1-1/\wpn(H))\binom n2+O(n)}$, we see that 
only $o(|\Forb_H^n|)$ graphs have a partition as above in which $|a_i|\ge n^{1-\gamma}$ for any $i$.

We define a subset $X_i$ of $S_i$ as follows: starting from $S_i$, if there is  any induced subgraph $L$ of $H$ such that there do not exist $n^{1-\frac{\epsilon}{4}}$ disjoint sets of vertices in the set which induce $L$, we choose a maximal 
family of vertex-disjoint induced copies of $L$ and remove from our set all the vertices in all these copies.  We repeat
until for every induced subgraph $L$ of $H$, either there are no induced copies of $L$ left, or there are 
at least $n^{1-\frac{\epsilon}{4}}$ vertex-disjoint induced copies of $L$.  We then let $X_i$ be the remaining vertices.
We let $Z_i$ be the union of  $A_i$ and 
all the vertices we removed, so $X_i\cup Z_i = S_i\cup A_i$,
and we let $J_i$ be the set of induced subgraphs of $H$ that do not appear as induced subgraphs of $G[X_i]$. 
We note that for some constant  $C_H$ which depends on $H$, we have
$|Z_i|\le C_Hn^{1-\frac{\epsilon}{4}}$. So, for large $n$, (II) holds. 


Since  $S_i \cup A_i=X_i \cup Z_i$, we have  that   (IV) holds; and since (c) holds, so does (III).  
Furthermore, since $G[S_i]$ is $U(K)$-free so is $G[X_i]$. All that remains is to check (I).

We let ${\cal F}_i$ be the family consisting of those graphs containing none of the elements of $J_i$ as an induced subgraph.
Clearly,  each ${\cal F}_i$ is canonical.  Now, since we may assume $n$ is large, $\delta<\frac{1}{10\wpn(H)}$ and (IV)  and (II) hold, 
we see that we can assume that $G[S_i]$ is as large as we like and so each ${\cal F}_i$ is in fact really canonical. 
To complete the theorem we need only show that the proportion of graphs with the property that  $\{{\cal F}_1,\dots,{\cal F}_{\wpn(H)}\}$ does not 
certify $X_1,\dots,X_{\wpn(H)}$ is $o(|\Forb_H^n|)$. In order to do so, we  sum over  all possible choices for:
 (i) the original $S_i$ and $A_i$ satisfying (a),(b),(c); (ii) the  new partition $X_i,Z_i$ of $S_i \cup A_i$ for each $i$; and  (iii)  the 
subgraphs $G[S_i \cup A_i]$ which have the property of interest.

We know already that there are at most $2^{O(n^{2-\epsilon})}$ choices for the partition
and the graphs $G[A_i\cup S_i]$, and there are at most $2^{O(n)}$ choices for the sets $X_i$ and $Z_i$.


We now count, for a specific choice of partition and induced subgraphs $(G[S_1 \cup A_1],\dots,G[S_{\wpn(H)} \cup A_{\wpn(H)}])$, 
 the number of $G$ corresponding to this choice for which  $\{{\cal F}_1,\dots,{\cal F}_{\wpn(H)}\}$ does not 
certify $X_1,\dots,X_{\wpn(H)}$. This implies there is   a partition $Y_1,\dots,Y_{\wpn(H)}$ of $V(H)$ such that 
$H[Y_i]$ is in ${\cal F}_i$ and hence not in $J_i$. Thus, for each $i$, we can find 
$n^{1-\frac{\epsilon}{4}}$ disjoint sets of vertices in $S_i$ which induce $H[Y_i]$. 
We consider the complete $\wpn(H)$-partite graph whose vertices are these sets and where two sets are joined if they are in different $S_i$. 
An easy, standard argument (see, for instance, the proof  of Lemma 5 in \cite{LRSTT}) tells us that  we can find $\frac{(n^{1-\frac{\epsilon}{4}})^2}{100^{wpn(H)}}$ edge disjoint 
cliques in this graph. For each such clique, there is at least one choice  of  edges between the sets corresponding to its vertices  such that if  
$G$ makes this choice than $G$ contains $H$ as an induced subgraph. It follows that for some $\gamma_H$ which depends on $H$,
the number of $G$ corresponding to our choice of partition and  $(G[S_1 \cup A_1],\dots,G[S_{\wpn(H)} \cup A_{\wpn(H)}])$ is at 
most $2^{(1-\frac{1}{r}){n \choose 2}}2^{-\gamma_Hn^{2-\frac{\epsilon}{2}}}$.  
Since there are $2^{O(n^{2-\epsilon})}$ choices for the partitions and subgraphs $G[S_i\cup A_i]$, this 
gives a total of $2^{(1-\frac{1}{r}){n \choose 2}+O(n^{2-\epsilon})-\Omega(n^{2-\epsilon/2})}
=o(|\Forb_H^n|)$ possibilities, and the result follows.
\end{proof}

\section{A Second Near Witnessing Partition}

We want to strengthen Theorem \ref{Hstruc.thm} when $H=C_l$ for some $l \ge 3$   by reducing the size of the exceptional set
to a polynomial in $\log n$. We start the process in this section by obtaining further properties of the original 
partition which hold for such $H$.  We develop the tools to do so, which can be applied more broadly, in the first subsection.
We describe the further properties of the  partition we obtain for the $H$ which interest us in the second. 

\subsection{Atypical Neighbourhoods in Typical  Bad Graphs}

We say that an $H$-free graph is {\it good} if it has a really canonical  $H$-freeness witnessing $wpn(HG)$-partition. 
We say that an $H$-free graph is  {\it bad} if this is not the case. The key to our results is to show that  
for the $H$ we consider that the number of bad $H$-free graphs is $o(1)$  of the number of good  $H$-free graphs.  
We know that there are more than $2^{(1-\frac{1}{wpn(H)}){n \choose 2}}$ good  $H$-free graphs. 
So we can ignore a bounded number of families of bad graphs which have 
$o(2^{(1-\frac{1}{wpn(H)}){n \choose 2}})$ members.

We say that a bad  $H$-free  graph $G$ on $V_n$   is  $(\delta, \gamma,b)$-{\it compatible} 
with $\mathcal F_1,\dots,\mathcal F_{\wpn(H)}$
on  some partition  $\mathcal P=\{Y_1,...,Y_{wpn(H)}\}$  of $V_n$
if for some set $B$ of $b$ vertices  it has a partition  satisfying (I)-(IV)  with $Y_i=X_i \cup Z_i$ for each $i$.
Theorem \ref{Hstruc.thm}  tells us that for every $\delta>0$ sufficiently small in terms of $H$, there are $b$ and $\gamma$ 
such that  the family of  bad graph  not   $(\delta, \gamma,b)$-compatible with  some really canonical $H$-freeness witnessing partition   $\mathcal F_1,\dots,\mathcal F_{\wpn(H)}$ on some partition  of $V_n$  is $o(|{\cal F}^H_n|)$. 
Hence we need only bound the number of   the remaining bad graphs.

In order to bound the number of such bad graphs, we think of exposing an $H$-free  graph $G$ which is  
$( \delta, \gamma, b)$-compatible with some $ \mathcal F_1,\dots,\mathcal F_{\wpn(H)}$ on a specific partition $\cal P$ of $V_n$ and corresponding sets $Z_1,...,Z_{wpn(H)}$ via a two step process:
\begin{itemize}
 \item A choice $C_1$ for  the vertices of $Z_1 \cup \dots \cup Z_{\wpn(H)}$ and the edges inside the classes $Y_i$.
 \item A choice $C_2$ of the edges between distinct classes $Y_i$.
\end{itemize}


We can choose $Z_i$  and  the edges inside $Y_i$ that are incident with it by first choosing $B$ and $Z_i$,  then choosing  the  neighbourhoods of the $b$ vertices of $B$,  next assigning each vertex of $\cup_i Z_i$ to one of these $b$ vertices and  finally specifying the at most $\delta n$ vertices in the symmetric difference  of the neighbourhoods of these two vertices.  So the number of such choices is: 

\begin{equation}\label{upperv}
n^b\binom{n}{n^{1-\gamma}}2^{nb}b^{n^{1-\gamma}}\binom{n}{\delta n}^{n^{1-\gamma}}=o( 2^{\sqrt{\delta}n^{2-\gamma}}),
\end{equation}
for $\delta$ sufficiently small (note that $\binom{n}{\delta n}<2^{\sqrt \delta n/2}$ for small $\delta$).
Similarly, for sufficiently small $\delta$, the number of choices for edges inside $X_i$ that are incident with a given set $S\subset X_i$ is at most
\begin{equation}\label{yinpre}
n^b2^{nb}b^{|S|}\binom{n}{\delta n}^{|S|}=o(2^{(b+1)n}2^{\sqrt\delta n|S|}).
\end{equation}

For each such $G$ and ${\cal P}$,  for each  for $1 \le i \le wpn(H)$  we define  if possible an (ordered)   set  $M_i$ of $\lceil \frac{n}{8wpn(H)} \rceil$ disjoint edges 
of $G[X_i]$, and if possible   an ordered set  $N_i$ of $\lceil \frac{n}{8wpn(H)} \rceil$  stable sets of size three in $G[X_i]$. 

To make our choice unique for each $i$ with $1 \le i \le wpn(H)$  where the desired object exists we  grow a  matching (respectively set of triples)  in $G[X_i]$  which is initially empty by,   at each step,   adding  into the current  (possibly empty) matching edge (respectively triple)   the lexicographically smallest vertex not yet used such that after adding it, we can still extend the current choice to the desired matching  (respectively set of triples)  of size $\lceil \frac{n}{8wpn(H} \rceil$.   Since deleting the vertices of a maximum matching yields a stable set we see that for every $i$, at least one of $M_i$ or $N_i$ exists.   


We define the collection ${\cal A}={\cal A}(G,{\cal P})$ with respect to such a graph $G$ on $V_n$ which is  
$( \delta, \gamma, b)$-compatible with  some  $\mathcal F_1,\dots,\mathcal F_{\wpn(H)}$ on a specific 
${\cal P}$ as follows. 
We say that a subset $S$ of $V_n$ is {\em atypical} if
$|S| \le 6$ and  one of the following holds: 
\begin{enumerate} \item[(i)] for  some $Y_i$ disjoint from $S$ for which $M_i$ existe and some pair of subsets $D_1$ and $D_2$ of $S$, 
the number of elements $\{x,x'\} \in M_i$ such that  one of $x,x'$ has neighbourhood $D_1$ on $S$ and the other has neighbourhood $D_2$ on 
$S$ is at most $\frac{n}{2^{17}\wpn(H)}$; or
 \item[(ii)] for  some $Y_i$ disjoint from $S$  for which $N_i$ exists and some triple  of subsets $D_1,D_2$ and $D_3$  of $S$,
the number of triples  $N_i$ such that  the $i^{th}$ element of the triple has neighbourhood $D_i$ on $S$  is at most $\frac{n}{2^{22}\wpn(H)}$; 
\item[(iii)] for some $Y_i$ 
intersecting $S$ in $2$ vertices,   we have that the vertices of 
$Y_i \cap S$ have at least $\frac{n}{10\wpn(H)}$ common neighbours in $Y_i$
and for some subset $D$ of $S$ which contains $S \cap Y_i$ the number of vertices of $Y_i$ with this  neighbourhood on $S$ is at most 
$\frac{n}{50000\wpn(H)}$.
\item[(iv)]for some $Y_i$ 
intersecting $S$ in $2$ vertices,   we have that the vertices of 
$Y_i \cap S$ have at least $\frac{n}{10\wpn(H)}$ common nonneighbours in $Y_i$
and for some subset $D$ of $S$ which is disjoint from  $S \cap Y_i$ the number of vertices of $Y_i$ with this  neighbourhood on $S$ is at most 
$\frac{n}{50000\wpn(H)}$.
\end{enumerate}
We choose a collection  $\mathcal A$ of    pairwise disjoint atypical sets.
We make this  choice  unique by choosing the lexicographically minimal one  of the given size (i.e.~we minimize the
size of the lexicographically smallest member of ${\cal A}$, then the lexicographically second smallest, and so on).  
If   our procedure grows $\mathcal A$  so that it has size $ \lfloor n^{1-\gamma }\rfloor +1$ we  terminate the procedure. 
Otherwise $\mathcal A$ is  a  maximal collection  of pairwise disjoint atypical sets.  
We have a canonical choice of $\mathcal A$.

We let $A$ be the set of vertices in the elements of ${\cal A}$ (so $|A|\le6|\mathcal A|$).  
We refer to $A$ as the {\em atypical set}, and set
$$X_i'=Y_i-A.$$ 
Recall that $m$ is the number of pairs of vertices in distinct classes $Y_i$.  We write
\begin{equation}\label{m01}
 m=m_0+m_1,
\end{equation}
where $m_0=m_0(A)$ is the number of pairs joining distinct classes $Y_i$ and incident with $A$,
and $m_1=m_1(A)$ is the number of pairs joining distinct classes $X_i'$.

We now bound the typical size of $\cal A$.  

It will be helpful to bound  the number of possibilities for various choices of edges that give $A=C$ as the atypical set  for a specific subset $C$ of $V_n$. 

\begin{claim}\label{claimY}
For all $H$  there is a $\beta=\beta_H>0$,  and a $\delta_H>0$ such that for all $\delta <\delta_H$ the following holds. For any $b$ and $\gamma>0$,
let $\mathcal P=\{Y_1,...,Y_{wpn(H)}\}$ be a partition of $V_n$ , let  $C \subset V_n$ with $|C|=O(n^{1-\gamma/2})$ vertices, and let 
$(\mathcal F_1,\dots,\mathcal F_{\wpn(H)})$ be a really canonical $H$-freeness witnessing $\wpn(H)$-sequence.
Consider $H$-free graphs $G$ that are 
$( \delta, \gamma, b)$-{\it compatible} 
with  $ \mathcal F_1,\dots,\mathcal F_{\wpn(H)}$ on  $\mathcal P$. 

(a) The number of choices for edges that are incident with $C$ and lie inside some class $Y_i$ is 
\begin{equation}\label{yin}
o(2^{(b+1)n}2^{\sqrt\delta |C|n}).
\end{equation}

(b) For any choice of $C_1$,  given $C=A$ the number of choices for edges that are incident with $C$ and join distinct classes is 
\begin{equation}\label{yout}
o(2^{m_0-(\beta_H+\sqrt\delta)|C|n}). 
\end{equation}

(c) The number of choices for edges incident with $C$ given $C=A$ is 
\begin{equation}\label{yboth}
o(2^{m_0+(b+1)n-\beta_H |C|n}),
\end{equation}
provided $n$ is sufficiently large (uniformly over the choice of $\mathcal P$).


\end{claim}

\begin{proof}
The first bound follows from \eqref{yinpre}, and the last bound follows from the first two.  Thus it is enough to prove the second.

Consider a graph together with a partition compatible with $\mathcal P$. 
As before, we have a choice $C_1$ of the vertices $Z_i$ and the graphs $G[Y_i]$, and a choice
$C_2$ of the remaining edges between vertex classes.
We fix $\mathcal A$, and 
consider the adjacencies between elements of $\mathcal A$ and the sets $X_i$.
Since  every element $S$ of ${\cal A}$ has size at most six, we have:
\begin{itemize}
\item for any  $X_i$ disjoint from
$S$, if we choose the edges from $S$ to $X_i-A$ uniformly from all the possibilities and $M_i$ exists,    the expected number of 
elements of $M_i$  one of whose endpoints has neighbourhood $D_1$ on $S$ and the other of whose endpoints has  neighbourhood $D_2$  on  $S$ 
is at least  $\frac{|M_i|-|A|}{2^{12}}$, and  
\item for any  $X_i$ disjoint from
$S$, if we choose the edges from $S$ to $X_i-A$ uniformly from all the possibilities and $N_i$ exists  the expected number of 
triples of $N_i$  the $i^{th}$ element of which has neighbourhood $D_i$ on $S$  
is at least  $\frac{|N_i|-|A|}{2^{18}}$, and  
\item   for any  $X_i$    intersecting $S$ in two vertices whose common neighbourhood in $X_i$    has least $\frac{n}{10\wpn(H)}$ elements,  if we were to choose the edges between $S-X_i$ and $X_i-S-A$ uniformly then the expected number of 
vertices  of $X_i$ with a given neighbourhood $D$ on  $S$  which contains $S \cap X_i$ is at least $\frac{n-10|A|}{160\wpn(H)}$. 
\item   for any  $X_i$    intersecting $S$ in two vertices whose common nonneighbourhood in $X_i$    has least $\frac{n}{10\wpn(H)}$ elements,  if we were to choose the edges between $S-X_i$ and $X_i-S-A$ uniformly then the expected number of 
vertices  of $X_i$ with a given neighbourhood $D$ on  $S$  which is disjoint from  $S \cap X_i$ is at least $\frac{n-10|A|}{160\wpn(H)}$. 
\end{itemize}
For any partition $\mathcal Z$ of $A$ into sets of size at most $6$, and a fixed choice of $C_1$, we  
consider a uniform random choice of edges between classes $Y_i$ that are incident with $A$ and bound the probability that for the resulting graph $G$ the  
elements of $\mathcal Z$ are all atypical.  For this to be the case, each $S$ in ${\cal Z}$  satisfies one of (i), (ii), (iii) , or (iv)
for some $X_i$. For each $S$ we specify an index $i_S$ and let $E(S,i_S)$ be the 
event that one of (i),(ii),(iii) or (iv)  holds for $S$ and $X_{i_S}$. 
It follows from the usual Chernoff bounds\cite{C81} that 
$\mathbb P[E(S,i)]$ is at most $2^{-\Omega(n)}$.Furthermore these events are mutually independent.
So the probability these events all hold is at most $2^{\Omega(|C|n)}$. 
But for a fixed choice of $A$, the number of choices 
for $\mathcal Z$ is at most $4^{|A|}|A|!=2^{o(|A|n)}$.  The bound now follows.
\end{proof}

Claim \ref{claimY} implies the atypical set $A$ is usually not too large.

\begin{claim}\label{c21}
For every $H$ there is a $\delta_H$ such that for all  $\delta<\delta_H$, $\gamma$  and $b$, for each $i$ between $1$ and $wpn(H)$, 
if $\mathcal F_i$ is a hereditary family such that  there are $2^{o(l^{2-\gamma})}$ 
graphs with $l$ vertices in ${\cal F}_i$.  Then,  for any $b$, 
the  number of bad  $H$-free graphs $G$  for which there is a ${\cal P}$ 
such that $G$ is 
$( \delta, \gamma, b)$-{\it compatible}   with $ \mathcal F_1,\dots,\mathcal F_{\wpn(H)}$
on   $\mathcal P$, and $|{\cal A}|= \lfloor n^{1-\gamma }\rfloor +1$ is $o(2^{(1-\frac{1}{wpn(H)}){n \choose 2}})$.
\end{claim}

\begin{proof}
We bound the probability that the atypical set has size  $a=\lfloor  n^{1-\gamma/2}\rfloor+1$, for a specific partition ${\cal P}$.   

For each $\mathcal F_i$, by hypotheis  the number of ways of choosing $G[X_i]$ is at most 
$2^{o(n^{2-\gamma})}$.  By \eqref{upperv},
the number of choices for the $Z_i$ and edges inside the classes $Y_i$ that are incident with them is at most $2^{\sqrt\delta n^{2-\gamma}}$.
It follows from \eqref{yout} that the number of choices for $A$ and the edges incident with $A$ between classes is at most 
$2^{m_0+(b+2)n-\beta_H an}$, and so the number of choices for $A$ and the edges between classes is at most
$2^{m+(b+1)n-\beta_H an}$.  

So the total number of choices is at most
$$2^{m+(b+2)n-\beta_H an+o(n^{2-\gamma})+\sqrt\delta n^{2-\gamma}}.$$
Summing over all the $2^{O(n)}$ choices for $\cal P$, this is $o(2^m)=o(2^{(1-\frac{1}{wpn(H)}){n \choose 2}})$.
\end{proof}

In light of Claim \ref{c21}, we may restrict to graphs which are compatible with a partition 
$\mathcal P$ such that $|{\mathcal A}|\le n^{1-\gamma}$ (note that this implies
that we choose a set $\mathcal A$ that is maximal satisfying (i) and (ii) above).

\begin{claim}\label{c22}
Suppose  first that   for  each $i$ between $1$ and $wpn(H)$, 
$\mathcal F_i$ is a hereditary family such that  there are $2^{o(l^2)}$ 
graphs with $l$ vertices in ${\cal F}_i$. 
Suppose further that  ${\mathcal J}$ is a set of subgraphs of $H$ and $H_1,...,H_{wpn(H)}$ is a set of subgraphs  
each of which is a vertex or stable set of size at most 3, 
 such that  for every $i$ with $1 \le i \le wpn(H)$,
and every graph $J \in \mathcal J$  
we can partition $V(H)$ into sets $L_1,..,L_{wpn(H)}$ such that $H[L_i]$ is isomorphic to $J$,  
and for $j \neq i$, $H(L_j)=H_j$. 
Then, for any   $\gamma$, $b$  and sufficiently small positive  $\delta$,   the number of $H$-free graphs  $G$ on $V_n$ for which there is a partition ${\cal P}$  of $V_n$ 
such that (i) $G$ is 
$( \delta, \gamma, b)$-{\it compatible}  with $\mathcal F_1,\dots,\mathcal F_{\wpn(H)}$ on   $\mathcal P$, (i)  for every $j$, if $H_j$ is an edge then $M_j$ exists  and otherwise  $N_j$  exists,  (iii) ${\cal A}$ is maximal, 
and  (iv) for some $i$, $X_i^\prime$  contains a set  of  vertices 
inducing a graph in $\mathcal J $ is $o(2^{(1-\frac{1}{wpn(H)}){n \choose 2}})$.
\end{claim}

\begin{proof}

Let us fix $C_1$,the $M_j,N_j$ and $i$, and  consider 
the number of choices for $C_2$ such that $X_i^\prime$  contains a set  $T$  of  vertices 
inducing  a graph $J$ in $\mathcal J$. For any such $G[T]$,
there is a  partition of $H$ into $L_1,..,L_{wpn(H)}$ as in the hypotheses of the lemma. 
We fix an isomorphism $g$ from $G[T]$ to  $H[L_i]$.

$T$ is disjoint from $A$, so it is not atypical.  Thus 
for each $j\ne i$ we can choose a set $P_j$ of $\lfloor \frac{n}{2^{20}\wpn(H)} \rfloor$ disjoint subsets  from $M_{j}$ or $N_j$  such that, for each $p_j \in P_j$, $g$ extends to an isomorphism 
between $H[L_i \cup L_j]$ and $G[T \cup p_j]$.  For any choice of elements $(p_j)_{j\ne i}$ with $p_j\in P_j$ for each $j$, there is some way of choosing edges between the
$p_j$ such that $T\cup \bigcup_{j\ne i} p_j$ induces a copy of $H$.

So consider the complete $(l-2)$ partite graph whose vertex classes consist of the sets $p_j$, $j\ne i$. 
A standard argument (as in the proof of Theorem  \ref{Hstruc.thm})  tells us that  we can 
find $\Omega(n^2)$  edge disjoint cliques $K_{l-2}$ in this graph. For each such clique, there is a choice  of  edges 
between the subsets  corresponding to its vertices  such that if $G$ makes this choice then $G$ contains $H$ as an induced subgraph
(induced by $T$ and the subsets of vertices in the $p_j$, $j\ne i$, corresponding to the clique). 
It follows that for some $\mu$ which depends on $H$,
the number of  choices for $C_2$ is at 
most $2^{m-\mu n^2}$.  By \eqref{upperv} and hypothesis, the number of choices for  $C_1$ and ${\cal P}$   is $2^{o(n^2)}$, 
summing over all choices of $C_1$ we see that the number of graphs   
containing such a four vertex set is $o(2^m)$.
\end{proof}

In the same vein, we obtain:

\begin{claim}\label{c22new}
Suppose  first that for   for  each $i$ between $1$ and $wpn(H)$, 
$\mathcal F_i$ is a hereditary family such that  there are $2^{o(l^2)}$ 
graphs with $l$ vertices in ${\cal F}_i$. 
Suppose further that  ${\mathcal J}$ is a set of subgraphs of $H$ and $H_1,...,H_{wpn(H)}$ is a set of two vertex subgraphs  
 such that  for every $i,j$ with $1 \le i <j \le wpn(H)$,
and every graph $J \in \mathcal J$  
we can partition $V(H)$ into sets $L_1,..,L_{wpn(H)}$ such that $H[L_i \cup L_j]$ is isomorphic to $J$,  
and for $k \not\in \{i,j\}$, $H(L_k)=H_k$. 
Then, for any   $\gamma$, $b$  and sufficiently small positive $\delta$,   the number of $H$-free graphs  $G$ on $V_n$ for which there is a partition ${\cal P}$  of $V_n$ 
such that (i) $G$ is 
$( \delta, \gamma, b)$-{\it compatible}  with $\mathcal F_1,\dots,\mathcal F_{\wpn(H)}$ on   $\mathcal P$
 (ii)  for every $j$,  if $H_j$ is  an edge then  $M_j$ is nonempty while otherwise $N_j$ is nonempty, (iii) ${\cal A}$ is maximal, 
and  (iv) for some $i\neq j $, $X_i^\prime \cup X_j^\prime$  contains a set  of  vertices 
inducing a graph in $\mathcal J $ is $o(2^{(1-\frac{1}{wpn(H)}){n \choose 2}})$.
\end{claim}

\subsection{Reexamining Our Partition}

We  show now  that we can focus on bad graphs where  ${|\cal A}|$ is polylogarithmic  in $\log n$
and the $G[X'_i]$ have a very strong structure.

For $H=C_{2l}$, we say that a bad  $H$-free  graph $G$ on $V_n$   is  {\it strongly}  $( \delta, \gamma, b)$-{\it compatible} 
with  $ {\cal F}_1,\dots,{\cal F}_{\wpn(H)}$ on  some partition  $\mathcal P=\{Y_1,...,Y_{wpn(H)}\}$  of $V_n$ if in addition to being compatible the following hold for the $M_i,N_i$ and 
$\cal A$ 
introduced in  the last section:
\begin{enumerate}
\item if $l>3$ then every $G[X'_i]$ is $P_4$-free and contains no graph with 4 vertices and 2 edges
 \item if $l>4$ then  every $G[X'_i]$ contains no graph with four vertices and 1 edge, and for every $i \neq j$, $G[X'_i \cup X'_j]$ is $P_6$-free
\item if $l>5$ no $G[X'_i]$ contains an $E_4$.
\item if $l=4$  then for  $j,k \neq i$, ,if $M_i$ exists $G[X'_j \cup X'_k]$ is $P_6$-free while if $N_i$ exists $G[X'_j \cup X'_k]$ does not contain the disjoint union of 2 $P_3$s.
\item if $l=3$ and $G[X'_i]$ contains a $P_4$ then $M_{3-i}$ is empty,  $N_{3-i}$ is not,  $G[X'_{3-i}]$ is $P_4$ free, and $G[X'_i]$ is the complement of a graph of girth $5$. 
\item $|{\cal A}|=o((\log n)^2) $.
\end{enumerate}

\begin{claim}
\label{nc1}
For $l>3$, for every sufficiently small  $\delta>0$ we can choose a $b$ and a $\gamma$ such that the number of $C_{2l}$ free graphs  on $V_n$ which are not {\it strongly}  $(\ \delta, \gamma, b)$-{\it compatible}  with  some really canonical $H$-freeness witnessing sequence  $\mathcal F_1,\dots,\mathcal F_{\wpn(H)}$ on  some partition  $\mathcal P=\{Y_1,...,Y_{wpn(H)}\}$  of $V_n$ is a $o(1)$ proportion of the $C_{2l}$-free graphs. 
\end{claim}

\begin{proof}
Applying Theorem \ref{Hstruc.thm}  we  can choose $b$ and $\gamma$ so that the number of $C_{2l}$ free graphs which are not 
  $( \delta, \gamma, b)$-{\it compatible}  with  some  really canonical $H$-freeness witnessing sequence $ {\cal F}_1,\dots,{\cal F}_{\wpn(H)}$ on some  partition  $\mathcal P=\{Y_1,...,Y_{wpn(H)}\}$  of $V_n$ is a $o(1)$ proportion of the $C_{2l}$-free graphs.  Furthermore, reducing $\gamma$ maintains this property so we can assume $\gamma<\frac{1}{4}$. We focus on the remaining $H$-free 
  graphs. 
 
 By our choice of $\gamma$ and  our bounds on the growth rates of the hereditary families in the relevant witnessing sequences it follows that for sufficiently small $\delta$, 
 The hypotheses of Claims \ref{c21},\ref{c22} and \ref{c22new} apply. Thus, by Claim \ref{c21}, the number of  remaining bad graphs where we can make a choice of
 ${\cal A}$ which is not maximal is a $o(1)$ proportion of the $C_{2l}$-free graphs.  
 
 We recall that for $l>3$, for  every $a+b=l-2$, $C_{2l}$ can be partitioned into $a$ edges,$ b $ nonedges, and  any graph which is a  $P_4$  or has four vertices and 2 edges, . Hence,  
 Claim \ref{c22} implies that the number  of the remaining bad graphs 
 which have a partition where ${\cal A}$ is maximal but (1) does not hold  is a $o(1)$ proportion of the $C_{2l}$-free graphs.  
 
 Now, for $l>4$  we have both (i) for every $a+b=l-2$, $C_{2l}$ can be partitioned into $a$ edges,$ b $ nonedges, and any graph which  has four vertices and 1 edge,  and (ii) for every $a+b=l-3$, $C_{2l}$ can be partitioned into  $a$ edges, $ b $ nonedges, and a  $P_6$.   Furthermore, $C_8$ can be partitioned into a $P_6$ and an edge and the disjoint union of 2 $P_3$s and a non-edge.  Hence, Claims  \ref{c22}  and  \ref{c22new} imply that the number  of the remaining bad graphs   which have a partition where $|{\cal A}|$ is maximal but one of (2)  or (4) does not hold  is a $o(1)$ proportion of the $C_{2l}$-free graphs.

A similar argument shows that for $l>5$, the number  of the remaining bad graphs which have a partition where ${\cal A}$ is maximal but (3) does not hold  is a $o(1)$ proportion of the $C_{2l}$-free graphs.

If $l=3$,  Claim \ref{c22} implies that the number of  such bad graphs for which  ${\cal A}$ is maximal,
$G[X'_i]$ contains a $P_4$ and either  $M_{3-i}$  exists or $N_{3-i}$ exists and $G[X'_i]$ contains an $E_3$ or 
a $2K_2$  is a $o(1)$ proportion of the $C_{2l}$-free graphs. Now, if $G[X'_i]$ contains no $E_3$ then $M_i$ 
is nonempty and applying  Claim  \ref{c22}  once again, we obtain that the number  of the remaining bad graphs   which have a partition where ${\cal A}$ is maximal but (5) does not hold  is a $o(1)$ proportion of the $C_{2l}$-free graphs.

  We recall that the number of $P_4$ -free graphs on $n$ vertices is less than $(2n)^{2n}$. So for any $l$  the number of choices of the edges within the $G[X'_i]$ such that either (i) all these graphs are $P_4$ free, or (ii) $l=3$,  for some $i$  $G[X_i]$ is $P_4$-free and $G[X'_{3-i}]$ is the complement of a graph of girth 5, is at most  $(2n)^{2n}$ times the number of choices for the edges  within  the $Y_i$  of good graphs for which ${\cal P} $ is a certifying partition. So, since there are at most $2^n$ partitions, Equation (\ref{yboth}) implies that the number of bad $C_{2l}$-free graphs for which
  $|{\cal A}|=\omega(\log n^{3/2})$ and (i) or (ii) hold   is a $o(1)$ proportion of the $C_{2l}$-free graphs. 
\end{proof} 

\section{  A Stronger Partition  in Two Simple Cases} 

It remains to show that for $l \ge 3$, for every sufficiently small  $\delta>0$ we can choose a $b$ and a $\gamma$ such that the number of bad $C_{2l}$ free graphs  on $V_n$ which are  {\it strongly}  $(\ \delta, \gamma, b)$-{\it compatible}  with  some really canonical $H$-freeness witnessing sequence  $\mathcal F_1,\dots,\mathcal F_{\wpn(H)}$ on  some partition  $\mathcal P=\{Y_1,...,Y_{wpn(H)}\}$  of $V_n$ is a $o(1)$ proportion of the $C_{2l}$-free graphs. As a first step, we will 
show we can find near-witnessing partitions for which deleting a set $Bad$ of  $O(log n)^6$ vertices yields a graph which has a witnessing partition.
In this section we do so for one simple case, and completely handle a second. 

{\bf Case 1:} $l=3$

In this case, we prove the following result yielding a more useful partition

\begin{claim}
\label{C6greatclaim}
There are  $\gamma, \epsilon, Q>0$ such that for almost every $C_6$-free graph $G$  there is a set  $Bad$ of fewer than $Q$ vertices such that $V(G)-Bad$ can be partitioned into a stable set $X_2$  and a set 
$X_1$ which induces a graph of girth 5 in $\overline{G}$. Furthermore $|X_i| < \frac{n}{2}+n^{1-\frac{\gamma}{4}},  s(\overline{G[X_1]})< \lceil 3n^{9/10} \rceil$, $\Delta(\overline{G[X'_i]})< \frac{n}{\log \log n}$   and $\overline{G[X_1]}$ has  at least $\epsilon n^{3/2}$ edges  and  $G[X_1]$ contains a family ${\cal W}_1$ of   $\frac{\epsilon  n}{12}$ disjoint $P_3$s   and a family ${\cal W}_2$ of   $\frac{\epsilon  n}{12}$ disjoint $\overline{P_3}$s. 

\end{claim}

\begin{proof}

We know that the number of $C_6$ free graphs is $2^{\Omega(n^{3/2})}2^{{n \choose 2}/2}$. If neither $G[X'_1]$ or $G[X'_2]$  contains a  $P_4$ 
then the number of choices for these two graphs is less than $(2n)^{2n}$ so. by  (\ref{yin}) and the fact that $|{\cal A}| =o((\log n)^2)$, the desired result follows.

So,  we can assume that for some $i$, $G[X'_i]$ contains a $P_4$. The definition of strongly compatible  implies that 
$M_i$ is empty and $G[X'_i]$ is the complement of a graph of girth 5
and $G[X'_{3-i}]$ is $P_4$-free. Thus there are fewer than $2n^{2n}$ choices for $G[X'_{3-i}]$. So, applying Theorem \ref{belasam} there is an $\epsilon>0$ such that 
the number of such graphs for which  $\overline{G[X'_i]}$ has fewer than $\epsilon n^{3/2}$ edges is a $o(1)$ proportion of the  $C_6$ free graphs. So, 
we can assume this is not the case.  Claim \ref{girth5hd.claim}
implies that at most $3n/2$ of the edges  of $\overline{G[X'_i]}$ are incident to a vertex of degree greater than $\frac{3\sqrt{n}}{2}$. 
Hence for large $n$ we can greedily construct  both a family ${\cal W}_1$ of   $\frac{\epsilon  n}{12}$ disjoint $P_3$s   of $G[X'_i]$ every vertex of each of which has degree at 
most $\frac{3 \sqrt{n}}{2}$ in $\overline{G[X'_i]}$ and a family ${\cal W}_2$ of   $\frac{\epsilon  n}{12}$ disjoint $\overline{P_3}$s   of $G[X'_i]$ every vertex of each of which has degree at 
most $\frac{3 \sqrt{n}}{2}$ in $\overline{G[X'_i]}$. 

We let ${\cal A}'$ be a maximal set of disjoint triples of $X'_{3-i}$ each inducing a $P_3$ or a $\overline{P_3}$ and $A'$ 
be the set of vertices in these triples. We note that since $M_i$ does not exist, this implies $G[X'_{3-i}-A']$ is a stable set.  
Considering the sets ${\cal W}_1$,${\cal W}_2$,$A$ and $B$, and applying Equation \ref{yboth}
we see that for sufficiently small $\delta$,  the number of choices for  ${\cal A},{\cal A}'$, all the edges incident to $A' \cup A$ and the other edges between parts is at most $(nb)^{|A'|}2^{m+(b+1)n+|A'|\delta n-\Omega(|A|n|)}(1-\frac{1}{2^9})^\frac{\epsilon n|{\cal A}'|}{12}$.  For small enough $\delta$
this is $2^{m+(b+1)n-\Omega(n|A' \cup A|)}$. But now the number of choices for $G[X'_i-A']$ and $G[X'_{3-i}]$ is at most the number of choices of edges for  $G[D_1]$ and $G[D_2]$  for a good graph
with a partition into $D_1$ and $D_2$ with $||D_1|-|D_2|| \le 1$.

Since there are only $2^n$ partitions of $G$,for some constant $B$, the number of bad  $C_6$-free graphs for which $|A|+|A'|>B$ 
is a $o(1)$ proportion of the bad graphs.
Applying Observation  \ref{girth5hd.claim2}  and \ref{girth5hd.claim3} we also obtain that the  number of bad  $C_6$-free graphs for which either $s(\overline{G[X'_i]})> \lceil 3n^{9/10} \rceil$  or $\Delta(\overline{G[X'_i]})> \frac{n}{\log \log n}$ 
is a $o(1)$ proportion of the bad graphs.

Setting $Bad= A \cup A'$,  we are done.

\end{proof}
{\bf Case 2:} $l=4$ and some $M_i$ does not exist. 

In this case rather than get a more useful  partition we show the  number of such graphs is a $o(1)$ proportion of the $C_8$-free graphs. 

We note that since the $G[X'_i]$ are $P_4$-free, (\ref{yin}) implies the number of choices for the  edges of all the $G[X'_i]$ 
is at most $n^{2n+o(n)}$.

We know  $N_i$ exists. Now $C_8$ can be partitioned into 2 $E_3$ and either an edge or a nonedge,
So, Claim \ref{c22new} implies the proportion of $C_8$-free graphs for which some other $N_j$ exists 
is a $o(1)$ proportion of the $C_8$ free graphs. Hence we can assume this is not the case.

Thus, the definition of strongly compatible  ensures that for all $j \neq i$, 
$G[X'_i \cup X'_j]$ does not induce a $P_6$. Hence since $N_i$ exists  and $P_6$ can be partitioned into both 
2$E_3$ and a $E_3$ and a $\overline{P_3}$, we obtain that  for $k\neq i$, by the definition of strong compatability (specifically the maximality of ${\cal A}$),  $G[X'_k]$ is the complement of a matching. 

Furthermore, since $C_8$ can be  partitioned into 
 the disjoint union of two $P_3$s and a stable set of size 2, Claim \ref{c22new}    implies  that the number of bad graphs for which the   union of these two complements of a matching  
  contains the disjoint union of two $P_3$s is a $o(1)$ proportion of the $C_8$ free graphs. So we can assume this is not the case. 
   If either of the complements of the matching  contain $n^{2/3}$ nonsingleton components 
 then it contains $n^{2/3}$ disjoint $P_3$s. So if both complements of matchings contain $n^{2/3}$ nonsingleton components then the proportion of choices for the edges between the partition elements is $2^m2^{-\Omega(n^{4/3})}$. By our bound on the number of choices for the $G[X_i]$, we obtain that  the number of bad graphs for which this occurs is $o(1)$ of the $C_8$-free graphs.

 Otherwise, we can specify the singleton components of these two graphs  in $2^n$ ways. 
 We can  finish off specifying the two complements of a matching by specifying a matching on the remaining at 
 most $\frac{n}{3}+n^{1-\frac{\gamma}{4}}+n^{2/3}$ vertices. 
 Now, we can specify a matching  of size  $l$ by taking an ordering of its vertices, and pairing consecutive vertices.
 Furthermore permuting the order of the edges, and swapping the order of the vertices within the edges shows 
 each matching occurs at least $2^{l/2}(\frac{l}{2})!$ times. So, the total number of such matchings is at most $l^{l/2}$.

We recall that by the definition of strongly compatible, $G[X'_i]$ is $P_4$-free and each component is the disjoint union of a 
clique and a stable set. Since $M_i$ is empty, there is a cover  for $G[X'_i]$ containing at most  $\frac{n}{24}$ vertices,
and ther rest of $X'_i$ lies in one component  $K$ of $\overline{G[X'_i]}$  We can choose K (and hence the edges from K to $G[X'-K]$) in fewer than $2^n$ ways. We can choose the edges 
 within $K$ by specifying the clique and the stable set in $2^n$ ways.
 We can choose the edges of the $P_4$-free graph $G[X'_i-K]$ in at most $(n/12)^{n/12}$ ways. 
 
 Thus, applying (\ref{yin}), the total number of bad graphs of this type are at most $2^{2{n \choose 2}/3}2^{3n}n^{\frac{n}{6}+o(n)+\frac{n}{12}}$.
 Now, there are at least $2^{2{n \choose 2}/3}B_{n/3}= 2^{2{n \choose 2}/3}n^{(1+o(1))n/3}$ $C_8$ free graphs which can be partitioned into 2 
 cliques and a multipartite graph. So, the graphs we are considering are a $o(1)$ proportion of the $C_8$ free graphs and we need treat them no further.

\section{A Stronger Partition: The Remaining Cases}

 \subsection{Finding A Strong Core} 

For each $i$, we define the {\em core $\Co_i$}  of $X_i^\prime$ to be the vertex set of the subgraph of $G[X_i']$ obtained by deleting all of its universal vertices,
and the 
$2n^{\frac{1}{5(l-1)^2}}$ largest anticomponents. We say $X_i^\prime$ has a {\em strong core} if $\Co_i$ has size 
at least $\frac{n}{5(l-1)^2}$. 

\begin{claim}\label{strongcore}
For $l>3$ all sufficiently small $\delta$, and $\gamma$ sufficiently small in terms of $\delta$,  and  for any $b$  and really canonical $C_{2l}$-freeness witnessing sequence $\mathcal F_1,\dots,\mathcal F_{\wpn(H)}$
the following holds. 
There are $o({2^{(1-\frac{1}{(l-1)}){n \choose 2}}}B_{ \lceil \frac{n}{l-1} \rceil}) $  bad $C_{2l}$-free graphs  $G$ for which there is a partition ${\cal P}$ 
such that  $G$ is   {\it strongly} $( \delta, \gamma, b)$-{\it compatible} with  $\mathcal F_1,\dots,\mathcal F_{\wpn(H)}$ on  $\mathcal P$ and no $X_i^\prime$ has a strong core. 
\end{claim}

\begin{proof}
By  (1) in the definition of strongly  compatible, if  for some $l>3$, $G$ is strongly  $( \delta, \gamma, b)$-compatible with $\mathcal F_1,\dots,\mathcal F_{\wpn(H)} $ on $\mathcal P$ 
then each of the  components in the complement of a $G[X'_i]$  is, in $G$, the disjoint union of a clique and a stable set. So if $X'_i$ 
does not have a strong core then we can specify  the graph induced by each $X^\prime_i$ by: 
(a)  choosing its   universal vertices; (b) specifying for 
each of its remaining vertices whether it lies in any of the $t=\lceil 2n^{\frac{1}{5(l-1)^2}}\rceil$ largest anticomponents,
and (if so) which anticomponent it lies in and whether it is in the clique or the stable set of which that anticomponent is the disjoint union,  ($2t+1$ choices per vertex); and (c) specifying the graph 
induced by the core, for which there are at most $f_3(\frac{n}{5(l-1)^2})$ choices. 
So the number of choices 
for the edges of $G$  within the $X^\prime_i$ that do not have a strong core  for a specific choice of ${\cal P}$  is at most  
\begin{equation}\label{check1}
 2^n(2t+1)^nf_3\left(\frac{n}{5(l-1)^2}\right)^{l-1}. 
\end{equation}

Now, $B_j$ is at most $j^j$
and so (by Lemma \ref{finsize}) $f_3(j)$ is at most $2^jj^j$. 
Thus \eqref{check1} is at most 
$$2^{n\log t+\frac{n}{5(l-1)}\log(\frac{n}{5(l-1)^2})+o(n\log n)}
=2^{\frac{n}{5(l-1)^2}\log n+\frac{n}{5(l-1)}\log  n+o(n\log n)}.$$.

Furthermore,  
by \eqref{yboth} for a fixed choice of $A$, when the size $a$ of ${\cal A}$ is at most   $n^{1-\gamma} $.there are at most $2^{m+(b+1)n-\beta_H an}$  choices for the edges incident with $A$, and the other edges between the partition classes. 
Thus  summing over all choices of partitions $\mathcal P$  and ${\cal A}$,
the number of bad $H$-free graphs we are counting is $o( 2^{(1-\frac{1}{(l-1)}){n \choose 2}+\frac{3n \log n }{10(l-1)}})$.

But $ B_{j}\ge j^{(1+o(1))j}=2^{(1+o(1))j\log j}$ so the number of $C_{2l}$ free graphs on $n$ vertices is 
at least $2^{(1-\frac{1}{l-1}){n \choose 2}} 2^{(\frac{1}{l-1}+o(1))n\log n}$.
So we are done.


\end{proof}

\subsection{Exploiting A  Strong Core}

In this section we complete the proof that for any $l$, we can find near-witnessing partitions for which deleting a set $Bad$ of  $O(\log n)^6$ vertices yields a graph which has a witnessing partition.
We focus on $l>4$  in this section, as we handled $l=3$ in the last section. 
I.e. we show:

 \begin{claim}
 \label{greatclaim} 
For $l>3$  there is $\gamma>0$ such that  in  almost every $C_{2l} $-free graph $G$  there is a set  $Bad$ of fewer than $(\log n)^6$ vertices such that $V(G)-Bad$ can be partitioned  as follows:
\begin{enumerate}
\item if $l=4$ into two cliques  and a graph which is the join of graphs which are the disjoint union of a clique and a stable set. 
\item if $l=5$ into three cliques and a graph which is the join of  graphs which are stable sets or the disjoint union of a clique and a vertex. 
\item if $l>5$ into  $l-2$  cliques and a graph which is the join of  graphs which are stable sets of size three or the disjoint union of a clique and a vertex. 
\end{enumerate} 
Furthermore all the partition elements have size within $2n^{1-\frac{\gamma}{4}}$ of $\frac{n}{l-1}$ and  the partition element which is not a clique contains a core which contains $\frac{n}{l-1} -\frac{3n}{\sqrt{\ln \ln n}}$ vertices lying in  (nonsingleton) components of size at most  $(\log n)^3$.  
\end{claim}

\begin{proof}

We have already seen that we can focus on   bad $C_{2l}$-free $G$  which are  strongly $( \delta, \gamma, b)$-compatible  with some really canonical $C_{2l}$-freeness.  
witnessing sequence $\mathcal F_1,\dots,\mathcal F_{l-1}$ on some   partition  ${\cal P}$ 
for which  (i) there is an $X_{i^*}$ which has a strong core.
We note that for $l>4$ all but one of the $G[X_i]$ induce a clique. 
When $l=4$, we can restrict our attention to partitions where each $M_i$ exists by the result of the last section. 
Hence we can also insist that:   (ii) for all but one $i$,  $X_i$ induces a clique, (iii) if $l=4$ then 
every $M_i$ exists. We say such graphs are {\it $i^*$- very strongly}  $( \delta, \gamma, b)$-compatible   with $\mathcal F_1,\dots,\mathcal F_{l-1}$ on ${\cal P}$. 
 
 To count these graphs, for each choice of  $C_1$, we  count the number of choices of  $C_2$ which extend it to such a graph. 
 
 We note that every component of the core of $X'_{i^*}$ has at most $\frac{n^{1-\frac{1}{5(l-1)^2}}}{l-1}$ vertices. 

We construct  $A'$ by repeatedly adding sets $S$ of size at most $6$ which are disjoint from $X'_{i^*} \cup A' \cup A$ and such that 
the number of vertices of the core of  $X'_{i^*}$ which have some specific neighbourhood on $S$ is less than $\frac{n}{2^{12}(l-1)^2}$,

We continue until either $A'$ is maximal or $|A' \cup A| \ge \log n)^2$.  We can  and do make this choice unique by always adding the 
lexicographically smallest choice. Applying the Chernoff  bounds we obtain 

\begin{claim}
\label{A'claim}. Letting $m'$ be the number of 
pairs of vertices one in $A'$ and the other in some $X'_i$,  there is a $\zeta>0$  such that the number of choices for $A'$ and the edges of $G-A$ leaving $A'$ 
is $2^{m'+(b+1)n-\zeta |A'|n}$. 
\end{claim}

Since there are $2^{O(n \log n)}$ edges in a $P_4$ free graph, combining this with Claim \ref{claimY}(c), we have that 
the number of bad $H$-free graphs $i^*$-extremely $(\delta, \gamma, b)$ very strongly compatible  with $\mathcal F_1,\dots,\mathcal F_{\wpn(H)}$ on some ${\cal P}$ for which $|A+A'|>(\log n)^2/2$  
is $o(2^m)$,  so we can restrict our attention to graphs with  $|A+A'|<(\log n)^2/2$  and hence with  $A'$ maximal.
We let $X''_i=X'_i-A'$. 

We note that, by the definition of strong compatability, for every $i \neq {i^*}$,  the nonsingleton components of the complement of $G[X''_i]$ induce in $G$ the disjoint union of a clique and a non-empty stable set. 
Thus for any such component $K$ there is a vertex $c_K$ which is nonadjacent  in $G$   to all the other vertices of $K$. 


For every $i \neq i^*$, We let ${\cal K}_i$ be the set of nonsingleton components of  $G[X''-i]$
containing less than $\frac{2n}{5wpn(H)}$ vertices. 
We note that for any $K$ in ${\cal K}_i$,  since $c_K \not \in A'$ it is non-adjacent to at least $\frac{n}{2^{12}(l-1)^2}$ vertices of the core of $X'_{i^*}$.
This means there are at least $z=\frac{n^{1/5(l-1)^2}}{2^{12}(l-1)}$ components  of this core containing a non-neighbour of $c_K$
and hence  a set ${\cal T}$ of $z$ disjoint pairs of  nonadjacent vertices of the core each containing a nonneighbour of $c_K$.

Suppose first there is a  pair $P$ of ${\cal T}$ and a  vertex $w$ of $K-c_K$ such that $G[P \cup \{w,c_K\}]$ has exactly one edge.
Now since $A$ is maximal  and the number of common neighbours of $c_K$ and $w$ in $X''_i$ exceeds $\frac{2n}{5wpn(H)}$ 
there is a vertex $u$ of $X''_i$ which sees $w$ and $c_K$ but none of $P$. Since $A'$ is maximal there are at least $\frac{n}{2^{12}wpn(H)}$ 
vertices of $X'_{i^*}$ which miss all of $c_K,u,w$. Since $P$ is in the core of $X'_{i^*}$, most, and more importantly for us, at least one  of these 
vertices see both elements of $P$.
But we have obtained a $P_6$ in $X'_i \cup X'_{i^*}$, a contradiction. So this does not occur. 

It follows that for every vertex of $K-c_K$,  and every $P \in {\cal T}$, there is a neighbourhood on $P$ which is forbidden. 
Furthermore, the number of choices for $K$ (which determines the $K$ to $X''_i-K$ edges)  and the edges within it  is less than ${n \choose |K|}2^{|K|}$.
Since $|{\cal T} | \ge z$, for large $n$,  given a choice for $K$, the number of choices for   the edges from $K$ to the vertices in the pairs of ${\cal T}$ is at most 
$2^{2|{\cal T}||K|-|K|n^{1/6(l-1)^2}}$.  Thus, the number  of the bad graphs we are considering where there are more than 
$n^{1-\frac{1}{7(l-1)^2}}$ vertices in a ${\cal K}_i$ is  a $o(1)$ proportion of the good graphs.  Thus, we can restrict our attention to graphs where this is not  the case.

Now, there  are at most two non-singleton components of $G[X''_i]$ which are not in ${\cal K}_i$, 
 so we can partition the remainder of  each $X''_i$ up into its components 
in at most $3^n$ ways and having done so there are at most $2^n$ choices for the graph we have created.
So we see that in total the number of choices for the graph outside $G[X'_{i^*}]$ is at most $2^m2^{(b+4)n}$. 

This  implies that  we need only consider  graphs such that  $G[X'_{i^*}]$ has  a  much stronger structure.
For any partition of $l$ into parts, there at most $k$ graphs whose components in the complement 
are the given parts, and such that each induces the join of a clique and a stable set. So, Observations \ref{obs:bruce4}, \ref{obs:bruce5}   and \ref{obs:bruce9} imply that  for some constant $\alpha$ the subset of the bad graphs we are  counting 
 where $G[X'_{i^*}]$ has fewer than $\frac{n}{\alpha \ln n}$  nonsingleton components, more than $\frac{n}{\sqrt{\ln n}}$  singeleton components, or   more than $\frac{n}{\sqrt{\ln  \ln n}}$ vertices in  components of size exceeding $(\ln n)^3$  is a 
 $o(1)$ proportion of the $H$-free graphs. 
 
 Thus we can restrict our attention to counting bad graphs  such that $G[X'_{i^*}]$ has  more  than $\frac{n}{\alpha \ln n}$  
 nonsingleton components, fewer than   $\frac{n}{\sqrt{\ln n}}$  singleton components  and fewer  than $\frac{n}{\sqrt{\ln \ln n}}$ vertices in  components of size exceeding $(\ln n)^3$.  This   implies that $i^*$ must be the unique $i$ for which $X_i$
  is not a clique. It also implies that the  core of $X'_{i^*}$ contains  at least  $\frac{n}{l-1}-\frac{3n}{\sqrt{\ln \ln n}}$ vertices which are in   (nonsingleton) components 
  of size at most $(\log n)^3$. Thus,  since $A'$ is maximal,  for every $i \neq i^*$, for every vertex  $v$ of $X''_i$ there must be a family ${\cal T}$ of  at least $\frac{n}{2^{13}(l-1)(\log n)^3}$  nonadjacent pairs of vertices  of the core of $X'_{i^*}$ such that $v$ sees at most  one vertex of the pair.

 
We construct  $A''$ by repeatedly adding  a  vertex $v$ outside $X'_{i^*} \cup A \cup A'$ for which there is no pair of nonadjacent vertices of the
core of $X'_{i^*}$ exactly one of which is a neighbour of $v$.We note that we can specify the edges from $v$ to the core of $X'_{i^*}$ by specifying the choices of the edge to $c_J$ for each 
component $J$ of its complement, so by the lower bound on the size of the core, and the upper bound on the number of its complement's components, there are $2^{o(n)}$ choices for the edges  from any $v$ added to $A''$ to $X'_{i^*}$. 
It follows that the number  of  bad graphs  we are counting for which $|A''|=\omega(1)$ is a $o(1)$ proportion of all $H$-free graphs and for some $Q$ we can and do  only count graphs 
for which $|A''|<Q$.

We let $X'''_i=X''_i-A''$.

 We can now make the same argument as above to show that  for  any $i \neq {i^*}$ and nonsingleton component $K$ of $G[X'''_i]$ with at most $\frac{2n}{5(l-1)}$ vertices,  the number of choices for $K$ and  the edges from $K$ to the vertices in the pairs of $P$ is at most 
$2^{2|{\cal P}||K|-\frac{n|K|}{(\log n)^4}}$.  Thus, the proportion of the bad graphs we are considering where there are more than 
$(\log n)^5$ vertices in a ${\cal K}_i$ is $o(1)$.

 It remains to consider the possibility that for some $i \neq i^*$ there is a component of $G[X'''_i]$ containing more than $\frac{2n}{5wpn(H)}$ vertices. 
  By the maximality of $A''$, there is  a pair $P$ of   nonadjacent pairs of vertices  of the core of $X'_{i^*}$ such  that $c_K$ sees exactly  one vertex of the pair.
  Since $A$ is maximal   there are at least $\frac{n}{50000(l-1)}$ nonneighbours of $c_K$ in $Y_i$
  which  have any specific neighbourhood on $P$. Most of these lie in  the clique $G[X_i]$ which has $\frac{(1+o(1))n}{l-1}$ vertices.
  So, we can choose adjacent nonneighbours of $c_K$ in this clique,   one of which $u$  sees none of $P$, and the other of which $v$  sees only one vertex  of $P$, that   not seen by $c_K$. 
  Since   $A'$ is maximal there are at least $\frac{n}{2^{12}(l-1)^2} $vertices of $X'_{i^*}$ which see none of $\{c_K,u,v\}$. But now 
  $G[X'_{i^*} \cup X''_i]$ contains a $P_6$,  a contradiction.

Now, setting $Bad= \cup_{i \neq i^*}{\cal K}_i \cup A \cup A' \cup A''$, we are done.  
\end{proof}
 
\section{Counting  Good Graph-Partition Pairs}
\label{counting} 

It remains to show that the number of bad $C_{2l}$-free graphs which have a partition 
as set out in Claims \ref{greatclaim} and \ref{C6greatclaim},is of smaller order than the number of good $H$-free graphs.
To some extent we will do this partition by partition, showing  the number of good graphs for which the given partition
certifies their goodness far exceeds  the number of bad graphs for which the given  partition certifies their near goodness (to define the  precise 
partition we consider for a bad graph, we will need to specify how to split $Bad$ between the partition elements which we do 
below).  The difficulty is that some good graphs may permit more than one good 
partition and hence we may overcount the good graphs (We don't care if we overcount the bad graphs).

We show now that for certain witnessing families no overcounting occurs. The key to doing so is to note that if 
for some  partition ${\cal P}$ 
we make a choice of $C_1$
which can be combined  with a choice of $C_2$ 
to yield a good $H$-free graph  such that ${\cal P}$ is an $H$-freeness witnessing partition certified by  $\mathcal F_1,\dots,\mathcal F_{l-1}$ then every choice of $C_2$ yields such a graph.

I.e. letting $m$ be the number of edges between the partition elements, there are $2^m$ choices for $C_2$ which can be combined with this
choice of $C_1$ to yield a good graph compatible with the partition.

\begin{claim}
\label{ourclaim2}
Let $p\ge1$ be fixed, and suppose that ${\mathcal F}_1,...,{\mathcal F}_p$ 
are families such that, for sufficiently large $l_0$,  we have 
\begin{enumerate}
\item either the graphs in
${\mathcal F}_1$  are all $P_4$-free or have girth at least five, and 
\item 
for $i>1$ and $l\ge l_0$, every graph in ${\mathcal F}_i$ on $l$ vertices  has minimum degree at least $\frac{31l}{32}+1$;
\end{enumerate}
Then  for every $\mu>0$,
for all sufficiently large $n$, the number of (graph, partition) pairs consisting of  
of a graph $G$ on $V_n$  and a partition of $V_n$ into  $X_1,...,X_p$ such that $G[X_i]$ is in ${\mathcal F}_i$ and 
$|X_i-\frac{n}{p}|< n^{1-\mu}$ for each $i$ is $\Theta(1)$ times the number of graphs on $V_n$ which  permit such a partition. 
\end{claim}

\begin{proof}
We focus on  (graph, partition) pairs where the following property holds:

 \begin{enumerate}
\item [\textbf{(P*):}]
\item[(a)] Every  two vertices of $G$ which are in $X_i$ for some $i>1$ 
have at least
$\frac{n}{4}+\frac{5n}{8p}$ common neighbours,
\item[(b)] two vertices in different  partition elements  have fewer than 
$\frac{n}{4}+\frac{5n}{8p}$ common neighbours, and
\item[(c)] for $i>1$, every  graph induced by the union of $X_i$  and a pair of vertices outside it contains a $P_4$, and a cycle of length at most four.   
\end{enumerate}

We will prove the following two assertions:
\begin{enumerate}
\item[$(\alpha)$] Every graph has  at most p partitions (up to order) into $X_1,...,X_p$ 
such that $G[X_i]$ is in ${\mathcal F}_i$ and $|X_i-\frac{n}{p}|< n^{1-\mu}$ for all $i$,   and  (P*) holds.
\item[$(\beta)$]  For large $n$, the following holds uniformly over partitions $X_1,\dots,X_p$
such that $|X_i-\frac{n}{p}|< n^{1-\mu}$ for every $i$:
the number of  graphs for which  $G[X_i]$ is in ${\mathcal F}_i$ for every $i$ is $\Theta(1)$ times 
the  number of graphs for which 
$G[X_i]$ is in ${\mathcal F}_i$ for every $i$ and in addition  (P*) holds. 
\end{enumerate} 
These two together prove our claim.  For let us say that a partition of a graph $G$ into 
$X_1,...,X_p$ is {\em good} if
$G[X_i]$ is in ${\mathcal F}_i$ and $|X_i-\frac{n}{p}|< n^{1-\mu}$ for all $i$, and {\em great} if it is good and also satisfies (P*).  
The number of graphs with a good partition is at least the number of graphs with a great partition, and by 
$(\alpha)$ this is  at least  $1/(p+1)!$ times the number
of great (graph,partition) pairs; and then $(\beta)$ shows that 
the number of great (graph,partition) pairs is at least some constant
times the number of good (graph,partition) pairs.  Thus the number of graphs with a good partition is
at least some constant times the number of good (graph,partition) pairs.  The reverse inequality is trivial.

Let us fixed a partition  $(X_1,\dots,X_{p})$ of $V(G)$ with
$|X_i-\frac{n}{p}|< n^{1-\mu}$ for all $i$. 
We let $C_1$ be all  choices of edges  within the partition elements 
such that $G[X_i] \in {\mathcal F}_i$.  We let $C_2$ be all choices of a set of edges between the partition elements. We note that 
there are $|C_1||C_2|$ graphs which permit this partition, as we can pair any choice from $C_1$ with  any choice from $C_2$.  

We first prove $(\alpha)$.
Consider a choice of $C_1$ and $C_2$  such that $G[X_i] \in {\mathcal F}_i$ for all $i$  and (P*) holds. 
For $i\ge2$, (a) and (b) tell us that for any other partition  $Y_1,...,Y_p$ of $G$ 
such that $G[Y_i] \in {\mathcal F}_i$  and (P*) holds,  $X_i$ lies within some $Y_j$.
 As $|Y_i-\frac{n}{p}|< n^{1-\mu}$, our bound on the size of the $X_i$ tell us they lie in distinct  $Y_j$. A symmetric argument tells us that each 
$Y_j$ for $j\ge2$ lies in some $X_i$. 

Now if  for $i,j>2$, $X_i$ lies in $Y_j$ then by (a) applied to the $Y_i$ and (b) applied to the $X_j$ we have $X_i=Y_j$. 
Reindexing we can assume that $X_i=Y_i$ for $p-2$ values of $i$ all greater than $1$. 
If we do not obtain exactly the same partition (up to order) then  $Y_1$ must  properly contains an $X_i$ for some $i>1$. 
Since  every graph in $\mathcal F_1$ is either $P_4$-free or has  a complement with girth at least  five, (c) tells us that $X_i=Y_1-v$ for some vertex $v$. 
Thus $Y_i=X_1-v$.  Applying  (a) to  $X_1-v$  and then (b) we see that $X_1-v$ must lie in the same part in any partition 
satisfying (P*).  So,  up to order the partition is fixed except for  the part containing $v$ .

It remains to show $(\beta)$.
To this end,
we note  that $C_1$ is actually an independent choice for each $i$ of  a graph on $X_i$ in ${\cal F}_i$.

For every choice of $C_1$, we can choose an element  from $C_2$ uniformly 
at random by choosing each edge joining vertices in different partition elements independently with probability $\frac{1}{2}$. 
We note that upon doing so, given a set of three vertices $u,v,w$  which is not contained in any $X_i$,
 the probability that $w$ is  a common neighbour of  $u$ and $v$  is at most $\frac{1}{2}$ if $w$  lies in the same partition element as 
one of $u$ or $v$ and exactly $\frac{1}{4}$ otherwise. 
By hypothesis, for large $n$ and $i\ge2$, $G[X_i]$ has minimum codegree $\frac{15|X_i|}{16}$.
So the expected number of common neighbours of two vertices 
is at most  $\frac{n}{4}+\frac{n}{2p}+n^{1-\mu}$ if they are in different partition elements and at least 
$\frac{n}{4}+\frac{45n}{64p}-n^{1-\mu}$ if they both lie in $X_i$ for some $i\ge 2$. 
So,  for every choice  of $C_1$, 
 $n \choose 2$ applications of the Chernoff Bound\cite{C81}, one for each pair of vertices, show that for large $n$ the proportion of choices  for  $C_2$ for which one of  (P*)(a) or (b) 
 fails is $O(1/n)$.

Finally, consider an $X_i$ for $i>2$  and two vertices $a$ and $b$ outside $X_i$. We partition $X_i$ into $\frac{|X_i|}{2}$ pairs of vertices. 
For each pair $\{c,d\}$, there is a choice of edges between this pair and $\{a,b\}$ for which these four vertices yield a $P_4$
(the choice depends on whether $ab$ and $cd$ are edges). Thus 
the probability that none of these sets of four vertices induces a $P_4$ is less than $(\frac{15}{16})^\frac{n}{3p}$. 
Since there are fewer than $n^2$ choices for $(a,b)$ and only $p$ choices for $X_i$, it follows that for every choice of $C_1$, 
the proportion of choices of $C_2$ for which (P*)(c) fails is $O(1/n)$. 
In the same vein there is a choice of edges between $\{a,b\}$ and $\{c,d\}$ yielding  a cycle within $G(\{a,b,c,d\})$ So a similar 
 argument shows that for all but a $O(1/n)$ fraction of choices for $C_2$, 
for each $i\ge 2$, the union of $X_i$ and any two vertices outside it  contains  a $C_3$ or $C_4$.

This proves our claim. 
\end{proof}

\section{Typical $C_6$-free graphs}
\label{C6}

In this section, we will prove Theorem \ref{C6.thm}.
By  Claim \ref{C6greatclaim} it is enough to show that the bad $C_6$ free graphs whose vertices have a partition into  $Bad,X_1,X_2$ with $|Bad|<Q$
as set out in that claim, are a $o(1)$ proportion of the $C_6$-free graphs. 

Our first task is to split the vertices of $Bad$ between $X_1$ and $X_2$. 
We choose a  constant $\mu<<\frac{1}{B}$ sufficiently small.  We  say that a vertex is {\it extreme} on $X_i$ if  $min(|N(v) \cap X_i|, |X_i-N(v)|)$ is not at least $\mu n$.
We note that  since $X_2$ is stable, and $\Delta( \overline{G[X_1]})<\frac{n}{\log \log n}, $ every vertex in $X_i$ is extreme on $X_i$.

Now, if there is a vertex $v$   in $Bad$ which is  not extreme on  $X_2$ and has more than $\frac{n}{\log \log \log n}$
nonneighbours in $X_1$ then  iteratively applying the fact that $S(\overline{G[X_1]})<3n^{9/10}$, , we can find 
$\frac{n}{3\log \log \log n}$ disjoint pairs of its nonneighbours  in $X_1$ such that the vertices of every pair are either nonadjacent or have a common nonneighbour.
Since $\Delta(\overline{G[X_1]})<\frac{n}{\log \log n}$, there is a family ${\cal K}$ of more than $ \frac{\log \log n}{4\log \log  \log n}$ disjoint subsets of $X_1$ 
such that for each $K \in {\cal K}$, $\overline{G[K \cup\{x\}]}$ is a cycle of length three or four. Since $v$ is not extreme on $X_2$, it follows that the number of choices for the edges 
from $K$ to $X_2$ is $2^{n/2-\Omega(n)}$ and hence the number of choices for the edges between $X_1$ and $X_2$ is $2^{\frac{{n \choose 2}}{2}-\omega(n)}$. 
Since the number of choices for $Bad$ and the edges leaving it is at most $2^{(Q+1)n}$ we see we  need only consider bad graphs for which this does not occur. 

So, we need only consider graphs where every vertex of $Bad$ is extreme  on at least one of $X_1$ and $X_2$.
We let $X^*_1$ be. the union of $X_1$ and the vertices of $Bad$ which are extreme on $X_1$ and $X^*_2=V-X^*_1$. 

By the results of the last section (applied to $G^*$,  it is enough to show that for
every choice of a partition  $X^*_1, X^*_2$, the number of bad graphs which have such a partition with this choice of $X^*_2$ is $o(1)$ of the
number of good graphs for which this is a certifying partition (where $X^*_2$ is stable and $X^*_1$ is the complement of a graph of girth at least 5). . 

In order to do so, we associate each  graph   with a choice of edges  within $X^*_1$. and $X^*_2$.
For a good graph  $G$ this choice is simply $G[X^*_1]$ and $G[X^*_2]$.  For a bad graph  it is obtained from $G[X_1]$ and $G[X_2]$ by adding edges 
so that the vertices of $Bad$  in $X^*_1$ are universal on $X^*_1$ and removing edges so the vertices of  $Bad$ in $X^*_2$
are isolated in $G[X^*_2]$. 

There are $2^{|X^*_1||X^*_2|}$ choices for a good graph which extends a given set of edges within a partition. 
We need to   show that there are $o(2^{|X^*_1||X^*_2|})$ choices of  bad graphs associated  to any such choice. 

There are  fewer than $(2n)^B$ choices for $Bad$, and since the vertices of $Bad$ are extreme in the partition. element which contains them,
given such a choice there are fewer than $2^B2^{B\sqrt{\mu }n}$ choices for the edges 
within the partition elements for bad graphs associated with the choice of edges for a specific good graph.

Now, if in the choice of edges  for a bad graph,
$X^*_1$ is not the complement of a graph of girth 5, then considering $\frac{n}{7}$ disjoint stable sets of size three in $X^*_2$, we see that there 
is some $\alpha$ such that there are   fewer than $2^{|X^*_1||X^*_2|-\alpha n}$ choices for the edges between $X^*_1$ and $X^*_2$
for a bad graph extending such a choice within the partition.  Since we are free to make $\mu<<\frac{\alpha^2}{B^2}$
we see we need only consider bad graphs where $X^*_1$ induces a graph  whose complement has girth  at least 5. 

In the same vein,   if in the choice for a bad graph,
$X^*_2$ is not a stable set, then since it is not a clique  it must contains a $P_3$ or an $\overline{P_3}$, so considering ${\cal W}_1$ and ${\cal W}_2$,   we see that there 
is some $\alpha$ such that there are  fewer than $2^{|X^*_1||X^*_2|-\alpha n}$ choices for the edges between $X^*_1$ and $X^*_2$ in a bad graph extending such a  choice within the partition. since we are free to make $\mu<<\frac{\alpha^2}{B^2}$
we see we need only consider bad graphs where $X^*_2$ is a stable set.

But a graph for which   $X^*_1$ induces a graph  whose complement has girth  at least 5 and  $X^*_2$ is a stable set, is good not bad.

\section{Typical $C_{2l}$-free graphs: $ l>3$}
\label{C2l}

In this section, we will prove Theorems \ref{C2l.thm}, \ref{C8.thm}, and \ref{C10.thm}.
By  Claim \ref{greatclaim}  it is enough to show that the bad $C_{2l}$ free graphs whose vertices have a partition into  $Bad,X_1,X_2,..X_{l-1}$ with $|Bad|<(\log n)^6$ and $X_i$ a clique for $i>2$
as set out in that claim, are a $o(1)$ proportion of the $C_{2l}$-free graphs. We label the partition elements so that for $i>1$, $X_i$ is a clique. 

Again our  first task is to split the vertices of $Bad$ between $X_1,X_2,...,X_{l-1}$. 
We again choose a  constant $\mu<<\frac{1}{B}$ sufficiently small.  We  say that a vertex is {\it extreme} on $X_i$ if  $min(|N(v) \cap X_i|, |X_i-N(v)|)$ is not at least $\mu n$.
We note that  since $X_i$ is either a clique or has a core which contains  at least $\frac{n}{l-1}-\frac{3n}{\sqrt{\log n}}$ vertices in components of the complement of 
size at most $(\log n)^3$,  every vertex of $X_i$ is extreme on $X_i$. 

Our first step is handle graphs  where some  vertex $v$ of $Bad$   is extreme on no  $X_i$.
We recall that  $X_1$ has a core which contains  at least $\frac{n}{l-1}-\frac{3n}{\sqrt{\log n}}$ vertices in components of the complement of 
size at most $(\log n)^3$. Thus,  the fact $v$ has at least $\mu n$ nonneighbours in $X_1$  means that we can find $\frac{\mu n}{3(l-1) (\log n)^3}$ triples of vertices 
from $X_1$ such that for each such triple $K$, $G[K \cup \{v\}]$ is a $P_4$ or the disjoint union of a $P_3$ and a vertex. 
Now, $C_{2l}$ can be partitioned into either of these graphs and $l=2$ ordered  edges, $\{e_j| j \in [l-1]-i\}$. 

For $j>1$, the fact that $v$ is not extreme on the clique $X_j$ means we can choose $\frac{\mu n}{2}$ ordered edges of $X_j$ such that $v$ has any required neighbourhood on 
each of these edges. So mimicing the proof of Claim \ref{c22},  it follows that the number of choices for the edges between the $X_i$ for which  such a $v$ exists 
 is $2^{(1-\frac{1}{l-1}){n \choose 2}-\omega(\frac{n^2}{(\log n)^6})}$. 
Hence we need only consider bad graphs for which this does not occur. 

So, we need only consider graphs where every vertex of $Bad$ is extreme  on at least one of $X_1$ and $X_2$.
We let $X^*_i$ be the union of $X_i$ and the vertices of $Bad$ which are extreme on $X_i$ but on none of $\{X_j|j<i\}$. 

By the results of  Section \ref{counting} it is enough to show that for
every choice of a partition  $X^*_1,...,X^*_{l-1}$, the number of bad graphs which have such a partition with is $o(1)$ of the
number of good graphs for which this is an $H$-freeness witnessing partition certified by $l-2$ families consisting of all cliques 
and one other family ${\cal F}_1$. 

In order to do so, we associate each  graph   with a choice of edges  within  every $X^*_i$.
For a good graph  $G$ this choice is simply the  $G[X^*_i]$.  For a bad graph  it is obtained from  the $G[X_i]$  by adding edges 
so that the vertices of $Bad$ are universal for the partition element they are in. 

Letting $m$ be the number of pairs of vetices not lying in the same partition element, there are $2^m$ choices for a good graph which extends a given set of edges within a partition. 
We need to   show that there are $o(2^m)$ choices of  bad graphs associated  to any such choice.

We now grow a family of disjoint  (atypical) sets  ${\cal A}'''$ each  element  $S$ consisting of either (i) of a vertex in some $X^*_i$  $i>1$ which is extreme on some
$X^*_j$,  $i \neq j$ or (ii) four vertices in $X^*_1$ which induce a graph $J$  such that for  some $j>1$ and subset  $N$ of $S$, the number of vertices of 
 $X_j$ which are adjacent to all of $N$ and none of $N-S$ is less than $\frac{n}{2^{20}(l-1)}$. We let $A''' =\cup\{S|S \in {\cal A}'''\}$.
 
 Since every vertex of $X^*_i$ is extreme on  $X^*_i$, there is a an $\alpha>0$ such that for any such $S$, the number of choices for edges incident to $S$ is 
 $2^{(1-\frac{1}{l-1}-\alpha)n|S|}$.  Since the $G[X_i]$ are $P_4$ free and $|Bad|< (\log n)^6$  It follows that we need only consider bad graphs for which $|A'''|<(\log n)^9$. 

Now, mimicing the proof of Claim \ref{c22} we can show that  the number of bad graphs for which 
 $X^*_1-A'''$ contains 4 vertices inducing a graph  $J$ such that $C_{2l}$ can be partitioned into $J$ 
and $l-2$ edges  is a $o(1)$ proportion of the $C_{2l}$-free graphs, so we can assume this is not the case.  

For each vertex $v$ in $X^*_i$ for some $i>2$ we  set $t_v=|X_i -N(v)|$ we choose $v^* \in V-X^*_1-A$  so as to maximize $t(v^*)$.
If $t(v^*)>2$ We can find $\frac{t(v^*)}{3}$ edges of $X_i$ such that $v$ has any required neighbourhood on 
each of these edges. 

We recall that  $X_1$ has a core which contains  at least $\frac{n}{l-1}-\frac{3n}{\sqrt{\ln \ln  n}}$ vertices in components of the complement of 
size at most $(\ln n)^3$. Thus,  the fact $v$ has at least $\mu n$ nonneighbours in $X_1$  means that we can find $\frac{\mu n}{3(l-1) (\ln n)^3}$ triples of vertices 
from $X_1$ such that for each such triple $K$, $G[K \cup \{v\}]$ is a $P_4$ or the disjoint union of a $P_3$ and a vertex. 
Now, $C_{2l}$ can be partitioned into either of these graphs and $l=2$ ordered  edges, $\{e_j| j \in [l-1]-i\}$. 

The fact that $v^*$ is not extreme on the clique $X_j$ , $j \neq i,1$ means  we can choose $\frac{\mu n}{2}$ ordered edges of $X_j$ such that $v$ has any required neighbourhood on 
each of these edges. 

Mimicing the proof of Claim \ref{c22},  it follows that the number of choices for the edges between the $X^*_i$ for which  such a $v$ exists 
 is $2^{m-\Omega( \frac{t(v^*)n}{(\log n)^3})}$. On the other hand the choices for the edges from $Bad$ within the $X^*_i$  is at most ${ n \choose t+|Bad|}^{|Bad|}$.
 Since $|Bad|<( \log n)^6$ it follows that we are done onless $t(v^*)<3$ and $A'''$ is empty.
 In this case there are $o(2^{(\log n)^{15}})$ choices for the vertices of $Bad$ and the edges from them which lie in the partition elements. 

We note that because of the strength of its core, $G[X_1]$ contains at least $\frac{n}{2 (\ln n)^3}$ disjoint $P_3$s. So, mimicing the proof of Claim \ref{c22new}, we obtain that  the number of choices for the edges between the partition elements for a choice of $C_1$ such that for some   
$j>1$, $G[X_j]$  contains a $P_3$
 or an $\overline {P_3}$ is $2^{m-\omega(\frac{n}{(\ln n)^3})}$. So we are done in this case, and we need only count bad graphs for which  no such $P_3$ or  
$\overline{P_3}$ exists. 

 But now since for $j>1$,$G[X_i]$ is a clique, so is $G[X^*_i]$.
 Mimicing the proof of Claim \ref{c22new}, we  also obtain that  the number of choices for the edges between the partition elements for a choice of $C_1$ where $G[X_i]$ is a clique for $i>1$ but 
 the partition does not show $G$ is good, is $2^{m-\omega(n)}$.

\end{document}